\documentclass{amsart} 
\usepackage{amscd}
\usepackage{amsmath,empheq}
\usepackage{amsfonts}
\usepackage{amssymb}
\usepackage{mathrsfs}
\usepackage[all]{xy}
\usepackage{mathtools}
\newtheorem{theorem}{Theorem}[section]

\newtheorem{lemma}[theorem]{Lemma}
\newtheorem{prop}[theorem]{Proposition}
\newtheorem{remark}{Remark}[section]
\newtheorem{corollary}[theorem]{Corollary}
\newtheorem{claim}{Claim}[section]

\newenvironment{proof-sketch}{\noindent{\bf Sketch of Proof}\hspace*{1em}}{\qed\bigskip}

\everymath{\displaystyle}
\newcommand{\RR}{\mathbb R}
\newcommand{\NN}{\mathbb N}

\renewcommand{\leq}{\leqslant}

\renewcommand{\geq}{\geqslant}
\begin{document}
\title[Nonlinear nonhomogeneous singular problems]{Nonlinear nonhomogeneous singular problems}
\author[N.S. Papageorgiou]{Nikolaos S. Papageorgiou}
\address[N.S. Papageorgiou]{ Department of Mathematics, National Technical University,
				Zografou Campus, Athens 15780, Greece \& Institute of Mathematics, Physics and Mechanics, 1000 Ljubljana, Slovenia}
\email{\tt npapg@math.ntua.gr}
\author[V.D. R\u{a}dulescu]{Vicen\c{t}iu D. R\u{a}dulescu}
\address[V.D. R\u{a}dulescu]{Faculty of Applied Mathematics, AGH University of Science and Technology, 30-059 Krak\'ow, Poland \& Institute of Mathematics, Physics and Mechanics, 1000 Ljubljana, Slovenia \& Department of Mathematics, University of Craiova, 200585 Craiova, Romania}
\email{\tt vicentiu.radulescu@imfm.si}
\author[D.D. Repov\v{s}]{Du\v{s}an D. Repov\v{s}}
\address[D.D. Repov\v{s}]{Faculty of Education and Faculty of Mathematics and Physics, University of Ljubljana \& Institute of Mathematics, Physics and Mechanics, 1000 Ljubljana, Slovenia}
\email{\tt dusan.repovs@guest.arnes.si}
\begin{abstract}
We consider a nonlinear Dirichlet problem driven by a nonhomogeneous differential operator
with a growth of order $(p-1)$ near $+\infty$
and with a reaction which has the competing effects of a parametric singular term and a $(p-1)$-superlinear perturbation which does not satisfy the usual  Ambrosetti-Rabinowitz condition. Using variational tools, together with suitable truncation and strong comparison techniques, we prove a ``bifurcation-type" theorem that describes the set of positive solutions as the parameter $\lambda$ moves on the positive semiaxis. We also show that for every $\lambda>0$, the problem has a smallest positive solution $u^*_\lambda$ and we demonstrate the monotonicity and continuity properties of the map  $\lambda\mapsto u^*_\lambda$.
\end{abstract}
\maketitle

\section{Introduction}

Let $\Omega\subseteq\RR^N$ be a bounded domain with a $C^2$-boundary $\partial\Omega$. In this paper, we study the existence and multiplicity of positive solutions for the following nonlinear, nonhomogeneous Dirichlet problem:
\begin{equation}
	-{\rm div}\,a(Du(z)) = \lambda \vartheta(u(z)) + f(z,u(z))\ \mbox{in}\ \Omega,\ u|_{\partial\Omega}=0,\ \lambda>0.
	\tag{$P_{\lambda}$}\label{eqp}
\end{equation}

In this problem, the map $a:\RR^N\rightarrow\RR^N$ involved in the definition of the differential operator is strictly monotone and continuous (hence maximal monotone, too) and satisfies certain additional regularity and growth conditions listed in hypotheses $H(a)$ below.
In particular, we assume that the differential operator has a growth of order $(p-1)$ near $+\infty$.
These hypotheses are general enough to incorporate in our framework many differential operators of interest such as the $p$-Laplacian and the $(p,q)$-Laplacian.

In the reaction of the problem (the right-hand side of $\eqref{eqp}$), we have the competing effects of two nonlinearities. One is the parametric term $\lambda\vartheta(x)$ with $\vartheta(\cdot)$ being singular at $x=0$. In the literature, we usually encounter the special case $\vartheta(x)=x^{-\beta},\ x\geq0$, with $\beta\in(0,1)$. The second term is the perturbation $f(z,x)$, which is a Carath\'eodory function (that is, for all $x\in\RR$ the map $z\mapsto f(z,x)$ is measurable and for almost all $z\in\Omega$ the map $x\mapsto f(z,x)$ is continuous).
We assume that $f(z,\cdot)$ exhibits $(p-1)$-superlinear growth near $+\infty$ but without satisfying the usual (in such cases) Ambrosetti-Rabinowitz condition (the AR-condition for short). Instead we employ a less restrictive condition which incorporates in our framework superlinear nonlinearities with ``slower" growth near $+\infty$, which fails to satisfy the AR-condition. So, in problem \eqref{eqp} we have the combined effects of two different types of nonlinearities: singular and superlinear. Our aim is to study the set of positive solutions of \eqref{eqp} and determine how this set changes as the parameter $\lambda$ moves along the positive semiaxis $(0,+\infty)$. In this direction, we prove a ``bifurcation-type" result which produces a critical parameter $\lambda^*>0$ such that the following properties hold:
\begin{itemize}
	\item for all $\lambda\in(0,\lambda^*)$ problem \eqref{eqp} has at least two positive smooth solutions;
	\item for $\lambda=\lambda^*$ problem \eqref{eqp} has at least one positive smooth solution;
	\item for all $\lambda>\lambda^*$ problem \eqref{eqp} has no positive solutions.
\end{itemize}

In addition, we show that for every $\lambda\in (0,\lambda^*]$ problem \eqref{eqp} has a smallest positive solution $u^*_\lambda$ and we establish the monotonicity and continuity properties of the map $\lambda\mapsto u^*_\lambda$.

Previously, such bifurcation type results have been proved by Sun, Wu \& Long \cite{32}, Haitao \cite{14}, Ghergu \& R\u adulescu \cite{gr1, gr2, 12} (semilinear problems driven by the Laplacian), and by Giacomoni, Schindler \& Taka\v{c} \cite{13}, Papageorgiou, R\u adulescu \& Repov\v{s} \cite{27}, Papageorgiou \& Smyrlis \cite{28}, Papageorgiou, Vetro \& Vetro \cite{29}, Papageorgiou \& Winkert \cite{30} (nonlinear problems driven by the $p$-Laplacian). In all these works the singular term has the special form $u^{-\beta}$, where $\beta\in(0,1)$. To the best of our knowledge, there are no works dealing with nonhomogeneous singular problems.

\section{Mathematical background and hypotheses}

Let $X$ be a Banach space and $X^*$ its topological dual. We denote by $\langle\cdot,\cdot\rangle$ the duality brackets for the pair $(X^*,X)$. Given $\varphi\in C^1(X,\RR)$, we say that $\varphi(\cdot)$ satisfies the ``Cerami condition" (the ``C-condition" for short), if the following property holds:
\begin{equation*}
	\begin{array}{ll}
		\mbox{``Every sequence}\ \{u_n\}_{n\geq1}\subseteq X\ \mbox{such that}\\
		\{\varphi(u_n)\}_{n\geq1}\subseteq\RR\ \mbox{is bounded}\
		\mbox{and}\ (1+||u_n||_X)\varphi'(u_n)\rightarrow0\ \mbox{in}\ X^*\ \mbox{as}\ n\rightarrow\infty,\\
		\mbox{admits a strongly convergent subsequence".}
	\end{array}
\end{equation*}

Evidently, this is a compactness-type condition on the functional $\varphi(\cdot)$. It is necessary since, in general, the ambient space $X$ is infinite-dimensional and therefore it is not locally compact. This condition is the main tool in the proof of a deformation theorem which leads to the minimax theory of critical values of $\varphi$. Prominent in this theory is the so-called ``mountain pass theorem" of Ambrosetti and Rabinowitz \cite{AR}, which we recall below.

\begin{theorem}\label{th1}
	Assume that $X$ is a Banach space, $\varphi\in C^1(X,\RR)$ satisfies the C-condition, $u_0,u_1\in X$, $||u_1-u_0||_X>\rho>0$,
	\begin{equation*}
		\max\{\varphi(u_0),\varphi(u_1)\}<\inf\left\{\varphi(u):||u-u_0||_X=\rho\right\}=m_\rho
	\end{equation*}
	and $c=\inf_{\gamma\in\Gamma}\max_{0\leq t<1}\varphi(\gamma(t))$ with $\Gamma=\{\gamma\in C([0,1],X):\gamma(0)=u_0, \gamma(1)=u_1\}$.
	Then $c\geq m_\rho$ and $c$ is a critical value of $\varphi$ (that is, there exists $u\in X$ such that $\varphi'(u)=0$ and $ \varphi(u)=c$).
\end{theorem}

We also refer to the monograph by Ambrosetti \& Prodi \cite{AP} for advances related to the mountain pass theory and to the survey by Pucci \& R\u adulescu \cite{pucrad} on the impact of this in nonlinear analysis.

The two basic spaces which we will use in the analysis of problem \eqref{eqp} are the Sobolev space $W^{1,p}_0(\Omega)$ and the Banach space $C^1_0(\overline\Omega)=\{u\in C^1(\overline\Omega):u|_{\partial\Omega}=0\}$. We denote by $||\cdot||$ the norm of the Sobolev space $W^{1,p}_0(\Omega)$. On account of the Poincar\'e inequality, we have
\begin{equation*}
	||u||=||Du||_p\ \mbox{for all}\ u\in W^{1,p}_0(\Omega).
\end{equation*}

The space $C^1_0(\overline\Omega)$ is an ordered Banach space with positive cone
\begin{equation*}
	C_+=\{u\in C^1_0(\overline\Omega):u(z)\geq0\ \mbox{for all}\ z\in\overline\Omega\}.
\end{equation*}

This cone has nonempty interior given by
\begin{equation*}
	{\rm int}\, C_+=\left\{u\in C_+:u(z)>0\ \mbox{for}\ z\in\Omega,\ \frac{\partial u}{\partial n}|_{\partial\Omega}<0\right\},
\end{equation*}
with $\frac{\partial u}{\partial n}=(Du,n)_{\RR^N}$, where $n(\cdot)$ is the outward unit normal on $\partial\Omega$.

We will also use the following open cone in $C^1(\overline\Omega)$:
$$
	{\rm int}\, \hat{C}_+=\left\{u\in C^1(\overline\Omega):u(z)>0\ \mbox{for all}\ z\in\Omega,\ \frac{\partial u}{\partial n}|_{\partial\Omega\cap u^{-1}(0)}<0\right\}.
$$

Let $l\in C^1(0,+\infty)$ with $l(t)>0$ for all $t>0$ and assume that
\begin{equation}\label{eq1}
	\begin{array}{ll}
		0<\hat{c}\leq \frac{l'(t)t}{l(t)}\leq c_0\ \mbox{and}\ c_1t^{p-1}\leq l(t)\leq c_2\,(t^{s-1}+t^{p-1})\\
		\mbox{for all}\ t>0,\ \mbox{with}\ c_1,c_2>0,\ 1\leq s<p<\infty.
	\end{array}
\end{equation}

The hypotheses on the map $a(\cdot)$ are the following:

\smallskip
$H(a):$ $a(y)=a_0(|y|)y$ for all $y\in\RR^N$ with $a_0(t)>0$ for all $t>0$ and
\begin{itemize}
	\item [(i)] $a_0\in C^1(0,+\infty),\ t\mapsto a_0(t)t$ is strictly increasing on $(0, +\infty),\ a_0(t)t\rightarrow 0^+$ as $t\rightarrow0^+$ and
		\begin{equation*}
			\lim_{t\rightarrow0^+}\frac{a'_0(t)t}{a_0(t)}>-1;
		\end{equation*}
	\item [(ii)] there exists $c_3>0$ such that
		\begin{equation*}
			|\nabla a(y)| \leq c_3\frac{l(|y|)}{|y|}\ \mbox{for all}\ y\in \RR^N\backslash\{0\};
		\end{equation*}
	\item [(iii)] $(\nabla a(y)\xi,\xi)_{\RR^N}\geq\frac{l(|y|)}{|y|}|\xi|^2$ for all $y\in\RR^N\backslash\{0\}$, $\xi\in\RR^N$;
	\item [(iv)] if $G_0(t)=\int^t_0 a_0(s)sds,$
	then there exists $q\in(1,p]$ such that $\limsup_{t\rightarrow0^+}\frac{q G_0(t)}{t^q}\leq c^*$ and
		\begin{equation*}
			0\leq pG_0(t)-a_0(t)t^2\ \mbox{for all}\ t>0.
		\end{equation*}
\end{itemize}

\begin{remark}\label{rem1}
	Hypotheses $H(a)(i),(ii),(iii)$ are dictated by the nonlinear regularity theory of Lieberman \cite{18} and the nonlinear maximum principle of Pucci and Serrin \cite{31}. Hypothesis $H(a)(iv)$ serves the needs of our problem, but it is very mild and it is satisfied in all cases of interest. Similar conditions have been recently used by Papageorgiou \& R\u{a}dulescu \cite{23bis,tranams,24, 25}. We mention that existence and multiplicity results for elliptic equations driven by nonhomogeneous differential operators, can be also found in the works of Colasuonno, Pucci \& Varga \cite{5}, Filippucci \& Pucci \cite{8}, and Papageorgiou \& R\u adulescu \cite{dcds2015}.
\end{remark}

Evidently, $G_0(\cdot)$ is strictly convex and strictly increasing. We set
\begin{equation*}
	G(y) = G_0(|y|)\ \mbox{for all}\ y\in\RR^N.
\end{equation*}

Then $G(\cdot)$ is convex, $G(0)=0$ and
\begin{equation*}
	\nabla G(y) = G'_0(|y|)\frac{y}{|y|}=a_0(|y|)y\ \mbox{for all}\ y\in\RR^N\backslash\{0\},\ \nabla G(0)=0.
\end{equation*}

Hence $G(\cdot)$ is the primitive of $a(\cdot)$. The convexity of $G(\cdot)$ and the fact that $G(0)=0$ imply that
\begin{equation}\label{eq2}
	G(y)\leq (a(y),y)_{\RR^N}\ \mbox{for all}\ y\in\RR^N.
\end{equation}

The main properties of the map $a(\cdot)$ are listed in the next lemma and follow easily from hypotheses $H(a)(i)(ii)(iii)$ and \eqref{eq1}.

\begin{lemma}\label{lem2}
	If hypotheses $H(a)(i)(ii)(iii)$ hold, then
	\begin{itemize}
		\item [(a)] the map $y\mapsto a(y)$ is continuous and strictly monotone (hence maximal monotone, too);
		\item [(b)] $|a(y)|\leq c_4\,(|y|^{s-1}+|y|^{p-1})$ for all $y\in\RR^N$ and some $c_4>0$;
		\item [(c)] $(a(y),y)_{\RR^N}\geq\frac{c_1}{p-1}|y|^p$ for all $y\in\RR^N$.
	\end{itemize}
\end{lemma}

By this lemma and \eqref{eq2}, we have the following growth estimates for the primitive $G(\cdot)$.

\begin{corollary}\label{cor3}
	If hypotheses $H(a)(i)(ii)(iii)$ hold, then $\frac{c_1}{p(p-1)}|y|^p\leq G(y)\leq c_5\,(1+|y|^p)$ for all $y\in\RR^N$ and some $c_5>0$.
\end{corollary}

The examples that follow show that hypotheses $H(a)$ are not restrictive and they incorporate in our framework many differential operators of interest (see Papageorgiou \& R\u{a}dulescu \cite{24}).

\medskip
{\sc Examples}:
(a) $a(y)=|y|^{p-2}y$ with $1<p<\infty$.
		The differential operator is the $p$-Laplacian defined by
		\begin{equation*}
			\Delta_p u={\rm div}\,(|Du|^{p-2}Du)\ \mbox{for all}\ u\in W^{1,p}_0(\Omega).
		\end{equation*}
	
(b) $a(y)=|y|^{p-2}y + |y|^{q-2}y$ with $1<q<p<\infty$.
		The differential operator is the $(p,q)$-Laplacian defined by
		\begin{equation*}
			\Delta_p u+\Delta_q u\ \mbox{for all}\ u\in W^{1,p}_0(\Omega).
		\end{equation*}
		
		Such operators arise in problems of mathematical physics and of mechanics, see Cherfils \& Ilyasov \cite{3} (reaction-diffusion systems) and Zhikov \cite{35} (homogenization of composites made of two different materials). Recently some existence and multiplicity results have been proved for such equations. In particular, we mention the works of Cingolani \& Degiovanni \cite{4}, Marano \& Mosconi \cite{20}, Marano, Mosconi \& Papageorgiou \cite{21}, Mugnai \& Papageorgiou \cite{22}, Papageorgiou \& R\u{a}dulescu \cite{23}, and  Sun, Zhang \& Su \cite{33}.

(c) $a(y)=(1+|y|^2)^\frac{p-2}{2}y$ with $1<p<\infty$.
		The differential operator is the generalized $p$-mean curvature differential operator defined by
		\begin{equation*}
			{\rm div}\,(1+|Du|^2)^\frac{p-2}{2}Du\ \mbox{for all}\ u\in W^{1,p}_0(\Omega).
		\end{equation*}

(d) $a(y)=|y|^{p-2}y\left(1+\frac{1}{1+|y|^p}\right)$ with $1<p<\infty$.
		The differential operator is
		\begin{equation*}
			\Delta_p u+{\rm div}\,\left(\frac{|Du|^{p-2}Du}{1+|Du|^2}\right)\ \mbox{for all}\ u\in W^{1,p}_0(\Omega).
		\end{equation*}

Such operators arise in problems of plasticity.

Let $A:W^{1,p}_0(\Omega)\rightarrow W^{-1,p'}(\Omega)=W^{1,p}_0(\Omega)^*$ $\left(\frac{1}{p}+\frac{1}{p'}=1\right)$ be the nonlinear map defined by
\begin{equation*}
	\langle A(u),h\rangle = \int_\Omega(a(Du),Dh)_{\RR^N}dz\ \mbox{for all}\ u,h\in W^{1,p}_0(\Omega).
\end{equation*}

The next result concerning the map $A(\cdot)$ can be found in Gasinski \& Papageorgiou \cite[p. 279]{10}, see Problem 2.192.

\begin{lemma}\label{lem4}
	If hypotheses $H(a)$ hold, then the map $A(\cdot)$ is bounded (that is, maps bounded sets to bounded sets), continuous, strictly monotone (hence maximal monotone, too) and of type $(S)_+$, that is,
	\begin{equation*}
		``u_n\xrightarrow{w}u\ \mbox{in}\ W^{1,p}_0(\Omega),\ \limsup_{n\rightarrow\infty}\langle A(u_n), u_n-u\rangle\leq0\Rightarrow u_n\rightarrow u\ \mbox{in}\ W^{1,p}_0(\Omega)."
	\end{equation*}
\end{lemma}

Next, we introduce the hypotheses on the singular nonlinearity $\theta(\cdot)$.

\smallskip
$H(\theta)$: $\vartheta:(0,+\infty)\rightarrow(0,+\infty)$ is a locally Lipschitz function with the following properties:
\begin{itemize}
	\item [(i)] there exists $\beta\in(0,1)$ such that
		\begin{equation*}
			0<c_6\leq\liminf_{x\rightarrow0^+}\vartheta(x)x^\beta\leq\limsup_{x\rightarrow0^+}\vartheta(x)x^\beta\leq c_7;
		\end{equation*}
	\item [(ii)] $\vartheta(\cdot)$ is nonincreasing.
\end{itemize}

In what follows, we set
\begin{equation*}
	\theta(x)=\int^x_0\vartheta(s)ds\quad \mbox{(the primitive of $\vartheta(\cdot)$)}.
\end{equation*}

\begin{remark}\label{rem2}
	Hypothesis $H(\theta)(i)$ implies that $\vartheta(\cdot)$ is singular at $x=0$. The function
	\begin{equation*}
		\vartheta(x)=cx^{-\beta}\ \mbox{for all}\ x>0\ \mbox{and for some}\ c>0,
	\end{equation*}
	which we encounter in the literature, satisfies hypotheses $H(\beta)$.
\end{remark}

Let $v_1,v_2\in {\rm int}\,C_+$ with $v_1\leq v_2$ and introduce the function
\begin{equation*}
	\gamma(z,x) = \left\{
		\begin{array}{ll}
			\vartheta(v_1(z))\ &\mbox{if}\ x<v_1(z)\\
			\vartheta(x)\ &\mbox{if}\ v_1(z)\leq x\leq v_2(z)\\
			\vartheta(v_2(z))\ &\mbox{if}\ v_2(z)<x.
		\end{array}
	\right.
\end{equation*}
This is a Carath\'eodory function. We set $\Gamma(z,x)=\int^x_0\gamma(z,s)ds$ and consider the functional $k_0:W^{1,p}_0(\Omega)\rightarrow\RR$ defined by
\begin{equation*}
	k_0(u) = \int_\Omega\Gamma(z,u(z))dz\ \mbox{for all}\ u\in W^{1,p}_0(\Omega).
\end{equation*}

Minor changes in the proof of Proposition 3 of Papageorgiou \& Smyrlis \cite{28} (in that paper, $\vartheta(x)=x^{-\beta}$ for all $x>0$ with $\beta\in(0,1)$), lead to the following result.

\begin{prop}\label{prop5}
	If hypotheses $H(\theta)$ hold, then $k_0\in C^1(W^{1,p}_0(\Omega),\RR)$ and $k'_0(u)=\gamma(\cdot,u(\cdot))$ for all $u\in W^{1,p}_0(\Omega)$
\end{prop}

We will also need strong comparison principles for singular problems.

\begin{prop}\label{prop6}
	Assume that hypotheses $H(a)$ and $H(\theta)$ hold, $\hat{\xi}\in L^\infty(\Omega)$, $\hat{\xi}(z)\geq0$ for almost all $z\in\Omega$, and $h_1,\, h_2\in L^\infty(\Omega)$ satisfy
	\begin{equation*}
		0<c_8\leq h_2(z)-h_1(z)\ \mbox{for almost all}\ z\in\Omega.
	\end{equation*}
Let	$u,v\in C^{1,\alpha}(\overline\Omega)\backslash\{0\}, \ 0<u(z)\leq v(z)$ for all $z\in\Omega$, $\lambda>0$, and for almost all $z\in\Omega$
	\begin{equation*}
		\begin{array}{ll}
			-{\rm div}\,a(Du(z))-\lambda\vartheta(u(z)) + \hat\xi(z)u(z)^{p-1} = h_1(z);\\
			-{\rm div}\,a(Dv(z))-\lambda\vartheta(v(z)) + \hat\xi(z)v(z)^{p-1} = h_2(z).
		\end{array}
	\end{equation*}
	Then $v-u\in {\rm int}\,\hat{C}_+$.
\end{prop}

\begin{proof}
	We have
	\begin{eqnarray}
		&&-{\rm div}\,(a(Dv(z)) - a(Du(z))) \nonumber \\
		&&=h_2(z)-h_1(z) + \lambda[\vartheta(v(z)) - \vartheta(u(z))] - \hat\xi(z)[v(z)^{p-1}-u(z)^{p-1}]\ \mbox{for almost all}\ z\in\Omega. \label{eq3}
	\end{eqnarray}

	Let $\{a_k(\cdot)\}^N_{k=1}$ be the components of the map $a(\cdot)$. For $y=(y_i)^N_{i=1},y'=(y'_i)^N_{i=1}\in\RR^N$	we have
	\begin{eqnarray}
		& a(y)-a(y')&=\int^1_0\frac{d}{dt}a(y'+t(y-y'))dt \nonumber \\
		&&=\int^1_0\nabla a(y'+t(y-y'))(y-y')dt\ \mbox{(by the chain rule)} \nonumber \\
		\Rightarrow & a_k(y)-a_k(y')&=\sum^N_{i=1}\frac{\partial a_k}{\partial y_i}(y'+t(y-y'))(y_i-y'_i)dt \ \mbox{for all}\ k\in\{1,...,N\}. \label{eq4}
	\end{eqnarray}
	
	Setting $\nabla_i=\frac{\partial}{\partial y_i}$ for $i\in\{1,...,N\}$, we introduce the following functions:
	\begin{equation*}
		\tilde{c}_{k,i}(z)=\int^1_0\frac{\partial a_k}{\partial y_i}(\nabla u(z)+t(\nabla v(z)-\nabla u(z)))dt\ \mbox{for all}\ k,i\in\{1,\dots,N\}.
	\end{equation*}
	
	On account of our hypothesis on $h_2-h_1$, we have that $\tilde{c}_{k,i}\in W^{1,\infty}(\Omega)$. We introduce the following second order differential operator in divergence form:
	\begin{equation}\label{eq5}
		L(w)=-{\rm div}\,(\sum^N_{i=1}\tilde{c}_{k,i}(z)\nabla_i w)=-\sum^N_{k,i=1}\nabla_k(\tilde{c}_{k,i}(z)\nabla_i w).
	\end{equation}
	
	Let $\eta=v-u\in C^{1,\alpha}(\overline\Omega)$. By hypothesis we have
	\begin{equation*}
		0\leq\eta(z)\ \mbox{for all}\ z\in\Omega,\ \eta\neq0.
	\end{equation*}
	
We	first suppose that for some $z_0\in\Omega$, we have $\eta(z_0)=0$, hence $v(z_0)=u(z_0)$. The functions
	\begin{equation*}
		x\mapsto x^{p-1}\ \mbox{and}\ x\mapsto\lambda\vartheta(x)
	\end{equation*}
	are uniformly continuous on $(0,+\infty)$ and on $(\mu,+\infty),\ \mu>0$, respectively. Since $\hat{\xi}\in L^\infty(\Omega)$ and $u(z_0)>0$, we can find small enough $\rho>0$ such that
	\begin{eqnarray}\label{eq6}
		h_2(z)-h_1(z)+\lambda[\vartheta(v(z))-\vartheta(u(z))] - \hat\xi(z)[v(z)^{p-1}-u(z)^{p-1}]\geq\frac{c_8}{2}>0 \\
		\mbox{for almost all}\ z\in B_\rho(z_0)=\{z'\in\RR^N:|z'-z_0|<\rho\}\subseteq\Omega. \nonumber
	\end{eqnarray}
	
	From \eqref{eq3}--\eqref{eq6} we see that
	\begin{equation*}
		L(\eta)(z)\geq \frac{c_8}{2}>0\ \mbox{for almost all}\ z\in B_\rho(z_0).
	\end{equation*}
	
	Invoking Theorem 4 of Vazquez \cite{34} (alternatively we can also use the Harnack inequality, see Pucci \& Serrin \cite[p. 163]{31}), we have
	\begin{equation*}
		\eta(z)>0\ \mbox{for all}\ z\in B_\rho(z_0),
	\end{equation*}
	a contradiction to the fact that $\eta(z_0)=0$. Therefore
	\begin{equation*}
		u(z)<v(z)\ \mbox{for all}\ z\in\Omega.
	\end{equation*}
	
	Next, let $z_0\in\partial\Omega$ be such that $\eta(z_0)=0$. Since by hypothesis, $\partial\Omega$ is a $C^2$-manifold, we can find small enough $d>0$ and an open $d$-ball $B_d$ such that
	\begin{equation*}
		B_d\subseteq\Omega\ \mbox{and}\ \partial\Omega\cap\partial B_d\ni z_0.
	\end{equation*}
	
	Then the Hopf boundary point theorem (see Theorem 5.5.1 of Pucci \& Serrin \cite[p. 120]{31} or Vazquez \cite[Theorem 5]{34}) implies that
	\begin{equation*}
		\frac{\partial \eta}{\partial n}(z_0)<0.
	\end{equation*}
	
	We conclude that $\eta=v-u\in {\rm int}\,\hat{C}_+$.
\end{proof}

The other strong comparison principle that we will use can be found in Papageorgiou \& Smyrlis \cite[Proposition 4]{28}. In \cite{28}, the differential operator is the $p$-Laplacian (that is, $a(y)=|y|^{p-2}y$ for all $y\in\RR^N$). With minor changes in the proof, as we did in the first part of the previous proof, we can obtain the result also for the more general differential operators of this paper. For more details on strong comparison principles involving nonlinear differential operators, we refer to Arcoya \& Ruiz \cite[Proposition 2.6]{2}, Gasinski \& Papageorgiou \cite[Section 3]{11} and Pucci \& Serrin \cite[Section 3.4]{31}.

First, we introduce the following notation. For functions $h_1,h_2\in L^\infty(\Omega)$ we write $h_1\prec h_2$, if for every compact $K\subseteq\Omega$, we can find $\eta_K>0$ such that
\begin{equation*}
	\eta_K\leq h_2(z)-h_1(z)\ \mbox{for almost all}\ z\in K.
\end{equation*}

Evidently, if $h_1,h_2\in C(\Omega)$ and $h_1(z)<h_2(z)$ for all $z\in\Omega$, then $h_1\prec h_2$.

\begin{prop}\label{prop7}
	Assume that hypotheses $H(\theta)$ hold, $\hat\xi\in L^\infty(\Omega)$, $\hat\xi(z)\geq0$ for almost all $z\in\Omega$, $h_1,h_2\in L^\infty(\Omega)$, $ h_1\prec h_2$, $u\in C_+$ with $u(z)>0$ for all $z\in\Omega$, $v\in {\rm int}\,C_+$, and for almost all $z\in\Omega$
	\begin{eqnarray*}
		-{\rm div}\,a(Du(z)) - \lambda\vartheta(u(z)) + \hat\xi(z)u(z)^{p-1} = h_1(z); \\
		-{\rm div}\,a(Dv(z)) - \lambda\vartheta(v(z)) + \hat\xi(z)v(z)^{p-1} = h_2(z).
	\end{eqnarray*}
	Then $v-u\in {\rm int}\,C_+$.
\end{prop}

\begin{remark}
	Note that by the weak comparison principle (see Damascelli \cite[Theorem 1.2]{6}), we have $u\leq v$.
\end{remark}

Another result which we will use, is the following one about ordered Banach spaces (see Gasinski \& Papageorgiou \cite[Problem 4.180, p.680]{10}).

\begin{prop}\label{prop8}
	Let $X$ be an ordered Banach space with positive (order) cone $K$ such that
	\begin{equation*}
		{\rm int}\, K\neq\emptyset\ \mbox{and}\ u_0\in {\rm int}\,K.
	\end{equation*}
	Then for every $u\in K$, we can find $\lambda_u>0$ such that $\lambda_uu_0-u\in K$.
\end{prop}

We will also need some facts about the first eigenvalue of $(-\Delta_\tau, W^{1,\tau}_0(\Omega))$, for $ 1<\tau<\infty$. So, we consider the following nonlinear eigenvalue problem:
\begin{equation}\label{eq7}
	-\Delta_\tau u(z)=\hat\lambda|u(z)|^{\tau-2}u(z)\ \mbox{in}\ \Omega,\ u|_{\partial\Omega}=0.
\end{equation}

We say that $\hat\lambda$ is an eigenvalue, if problem \eqref{eq7} admits a nontrivial solution $\hat{u}\in W^{1,\tau}_0(\Omega)$, known as an eigenfunction corresponding to the eigenvalue $\hat\lambda$. The nonlinear regularity theory for the $p$-Laplacian (see, for example, Gasinski \& Papageorgiou \cite[pp. 737-738]{9}), we deduce that $\hat{u}\in C^1_0(\overline\Omega)$. We know that there exists a smallest eigenvalue $\hat\lambda_1(\tau)$ having the following properties:
\begin{itemize}
	\item $\hat\lambda_1(\tau)$ is isolated (that is, we can find $\epsilon>0$ such that $(\hat\lambda_1(\tau),\hat\lambda_1(\tau)+\epsilon)$ does not contain any eigenvalue);
	\item $\hat\lambda_1(\tau)$ is simple (that is, if $\hat{u},\hat{v}$ are eigenfunctions corresponding to $\hat\lambda_1(\tau)$, then $\hat{u}=\xi\hat{v}$ for some $\xi\in\RR\backslash\{0\}$);
	\item we have
		\begin{equation}\label{eq8}
			\hat\lambda_1(\tau)=\inf\left\{\frac{||Du||^\tau_\tau}{||u||^\tau_\tau}:u\in W^{1,\tau}_0(\Omega)\right\}>0.
		\end{equation}
\end{itemize}

The infimum in \eqref{eq8} is realized on the corresponding one-dimensional eigenspace. From the above properties it follows that the elements of this eigenspace do not change sign. By $\hat{u}_1(\tau)$ we denote the positive, $L^\tau$-normalized (that is, $||\hat{u}_1(\tau)||_\tau=1$) eigenfunction. By the nonlinear maximum principle (see Vazquez \cite[Theorem 5]{34}), we have $\hat{u}_1(\tau)\in {\rm int}\,C_+$.

Now we will introduce some basic notation which we will use throughout this paper. So, let $x\in\RR$. We set $x^\pm=\max\{\pm x,0\}$. Then if $u\in W^{1,p}_0(\Omega)$ we define
\begin{equation*}
	u^\pm(z)=u(z)^\pm\ \mbox{for all}\ z\in\Omega.
\end{equation*}

We know that
\begin{equation*}
	u^\pm\in W^{1,p}_0(\Omega),\ u=u^+-u^-,\ |u|=u^++u^-.
\end{equation*}

If $u,v\in W^{1,p}_0(\Omega)$ and $u\leq v$, then
\begin{equation*}
	[u,v]=\{y\in W^{1,p}_0(\Omega):u(z)\leq y(z)\leq v(z)\ \mbox{for almost all}\ z\in\Omega\}.
\end{equation*}

Also, we denote by ${\rm int}_{C^1_0(\overline\Omega)}[u,v]$,  the interior in the $C^1_0(\overline\Omega)$-norm topology of $[u,v]\cap C^1_0(\overline\Omega)$ and $[u)=\{y\in W^{1,p}_0(\Omega):u(z)\leq y(z)\ \mbox{for almost all}\ z\in\Omega\}$.

Recall that $p'$ is the conjugate exponent of $p>1$, that is, $\frac{1}{p}+\frac{1}{p'}=1$. We denote  by $p^*$ the critical Sobolev exponent corresponding to $p>1$, that is,
\begin{equation*}
	p^*=\left\{
		\begin{array}{ll}
			\frac{Np}{N-p} &\mbox{if}\ N>p\\
			+\infty &\mbox{if}\ N\leq p.
		\end{array}
	\right.
\end{equation*}

Also, if $X$ is a Banach space and $\varphi\in C^1(X,\RR)$, then $K_\varphi=\{u\in X:\varphi'(u)=0\}$ (the critical set of $\varphi$).

Finally, we introduce the hypotheses on the perturbation term $f(z,x)$.

\smallskip
$H(f):$ $f:\Omega\times\RR\rightarrow\RR$ is a Carath\'eodory function such that $f(z,0)=0$ for almost all $z\in\Omega$ and
\begin{itemize}
	\item [(i)] $f(z,x)\leq a(z)\,(1+x^{r-1})$ for almost all $z\in\Omega$ and all $x\geq0$ with $a\in L^\infty(\Omega)$, $p<r<p^*$;
	\item [(ii)] if $F(z,x)=\int^x_0 f(z,s)ds$, then $\lim_{x\rightarrow+\infty}\frac{F(z,x)}{x^p}=+\infty$ uniformly for almost all $z\in\Omega$;
	\item [(iii)] if for every $\lambda>0$, we define $$e_\lambda(z,x)=\lambda[\vartheta(x)x - p\theta(x)] + [f(z,x)x-pF(z,x)],$$ then there exists $\hat\beta_\lambda\in L^1(\Omega)$ such that the map $\lambda\mapsto\hat\beta_\lambda$ is nondecreasing and
		\begin{equation*}
			e_\lambda(z,x)\leq e_\lambda(z,y) + \hat\beta_\lambda(z)\ \mbox{for almost all}\ z\in\Omega,\ \mbox{and all}\ 0\leq x\leq y;
		\end{equation*}
	\item [(iv)] for every $s>0$, there exists $\tilde\eta_s>0$ such that
		\begin{equation*}
			\tilde\eta_s\leq f(z,x)\ \mbox{for almost all}\ z\in\Omega,\ \mbox{and all}\ x\geq s
		\end{equation*}
		and
		\begin{equation*}
			\lim_{x\rightarrow0^+}\frac{f(z,x)}{x^{p-1}}=0\ \mbox{uniformly for almost all}\ z\in\Omega;
		\end{equation*}
	\item [(v)] for every $\rho>0$, there exists $\hat\xi_\rho>0$ such that for almost all $z\in\Omega$ the function
		\begin{equation*}
			x\mapsto f(z,x) + \hat\xi_\rho x^{p-1}
		\end{equation*}
		is nondecreasing on $[0,\rho]$.
\end{itemize}

\begin{remark}
	Since we are looking for positive solutions and the above hypotheses concern the positive semiaxis $\RR_+=[0,+\infty)$,  we may assume without any loss of generality that
	\begin{equation}\label{eq9}
		f(z,x)=0\ \mbox{for almost all}\ z\in\Omega,\ \mbox{and all}\ x\leq0.
	\end{equation}
\end{remark}

Hypotheses $H(f)(ii),(iii)$ together with hypothesis $H(\vartheta)(ii)$ imply that
\begin{equation*}
	\lim_{x\rightarrow+\infty}\frac{f(z,x)}{x^{p-1}}=+\infty\ \mbox{uniformly for almost all}\ z\in\Omega.
\end{equation*}

So, the perturbation term $f(z,\cdot)$ is $(p-1)$-superlinear near $+\infty$. However, we do not use the standard (in such cases) AR-condition. Recall that the AR-condition (unilateral version because of \eqref{eq9}), says that there exist $M>0$ and $\tau>p$ such that
\begin{equation}
	0<\tau F(z,x)\leq f(z,x)x\ \mbox{for almost all}\ z\in\Omega,\ \mbox{and all}\ x\geq M
	\tag{$10a$}\label{eq10a}
\end{equation}
\begin{equation}
	0<\underset{\Omega}{\rm essinf}\, F(\cdot,M).
	\tag{$10b$}\label{eq10b}\stepcounter{equation}
\end{equation}

Integrating \eqref{eq10a} and using \eqref{eq10b}, we produce the following weaker condition:
\begin{eqnarray*}
	& c_9x^\tau&\leq F(z,x)\ \mbox{for almost all}\ z\in\Omega,\ \mbox{and all}\ x\geq M, \\
	\Rightarrow & c_9x^{\tau-1}&\leq f(z,x)\ \mbox{for almost all}\ z\in\Omega,\ \mbox{and all}\ x\geq M\ \mbox{(see \eqref{eq10a})}.
\end{eqnarray*}

So, the AR-condition dictates that the perturbation $f(z,\cdot)$ has at least $(\tau-1)$-polynomial growth near $+\infty$. Here, instead of the AR-condition, we use the quasimonotonicity requirement on $e_\lambda(z,\cdot)$ stated in hypothesis $H(f)(iii)$. This hypothesis is a slightly more general version of a condition used by Li \& Yang \cite{19}. It is satisfied if, for example, we can find $M>0$ such that for almost all $z\in\Omega$ the function
\begin{equation*}
	x\mapsto\frac{\lambda\vartheta(x)+f(z,x)}{x^{p-1}}\ \mbox{is nondecreasing on}\ [M,+\infty).
\end{equation*}

Hence hypothesis $H(f)(iii)$ is less restrictive than the AR-condition and it also agrees with $(p-1)$-superlinear nonlinearities with ``slower" growth near $+\infty$.
For example, taking the singularity to be $cx^{-\beta}$ or, more generally, $\vartheta(x)=x^{-\beta}(1\pm\sin x)$, $\vartheta(x)=x^{-\beta} (c\pm\cos x)$ for $c>1$, then the following function satisfies hypotheses $H(f)$ but not the AR-condition (for the sake of simplicity we drop the $z$-dependence):
\begin{equation*}
	f(x) = x^{p-1}\ln(1+x)\ \mbox{if}\ x\geq0.
\end{equation*}

On the other hand, the common  superlinear nonlinearity
\begin{equation*}
	f(x)=x^{\tau-1}\ \mbox{if}\ x\geq0
\end{equation*}
satisfies both hypotheses $H(f)$ (with the above $\vartheta(\cdot)$) and the AR-condition.

We introduce the following two sets:
\begin{eqnarray*}
	&\mathcal{L}=\{\lambda>0:\ \mbox{problem \eqref{eqp} admits a positive solution}\},\\
	& S_\lambda= \mbox{the set of positive solutions of problem \eqref{eqp}}.
\end{eqnarray*}

Also, we define
\begin{equation*}
	\lambda^*=\sup\mathcal{L}.
\end{equation*}

We summarize our analysis of Problem \eqref{eqp} in the next theorem, which is the main result of the present paper.

\begin{theorem}\label{th21}
	If hypotheses $H(a),H(\vartheta),H(f)$ hold, then there exists $\lambda^*>0$ such that
	\begin{itemize}
		\item[(a)] for all $\lambda\in(0,\lambda^*)$ problem \eqref{eqp} has at least two positive solutions
		$$u_0,\hat{u}\in {\rm int}\, C_+,\ u_0\neq\hat{u};$$
		\item[(b)] for $\lambda=\lambda^*$ problem \eqref{eqp} has at least one positive solution $\hat{u}_{\lambda^*}\in {\rm int}\, C_+$;
		\item[(c)] for all $\lambda>\lambda^*$ problem \eqref{eqp} has no positive solutions;
		\item[(d)] for all $\lambda\in\left(0,\lambda^*\right]$ problem \eqref{eqp} has a smallest positive solution $u^*_\lambda\in {\rm int}\, C_+$ and the minimal solution map $\chi:\left(0,\lambda^*\right]\rightarrow C^1_0(\overline{\Omega})$ defined by $\chi(\lambda)=u^*_\lambda$ is strictly increasing and left continuous.
	\end{itemize}
\end{theorem}

\section{Positive solutions}

In this section we prove the bifurcation-type property described in Theorem \ref{th21}.

We first show that $\mathcal{L}\neq\emptyset$. To this end, let $\eta\in(0,1]$ and consider the following Dirichlet problem:
\begin{equation}\label{eq11}
	-{\rm div}\,a(Du(z))=\eta\ \mbox{in}\ \Omega,\quad u|_{\partial\Omega}=0.
\end{equation}

\begin{prop}\label{prop9}
	If hypotheses $H(a)(i)(ii)(iii)$ hold, then problem \eqref{eq11} has a unique solution $\tilde{u}_\eta\in {\rm int}\,C_+$, the map $\eta\mapsto\tilde{u}_\eta$ from $(0,1]$ into $C^1_0(\overline\Omega)$ is strictly increasing (that is, $0<\mu<\eta\leq1\Rightarrow\tilde{u}_\eta-\tilde{u}_\mu\in {\rm int}\,C_+$) and if $\eta\rightarrow0^+$, then $\tilde{u}_\eta\rightarrow0$ in $C^1_0(\overline\Omega)$.
\end{prop}

\begin{proof}
	The map $A:W^{1,p}_0(\Omega)\rightarrow W^{-1,p'}(\Omega)$ is maximal monotone (see Lemma 4) and coercive (see Lemma 2(c)), hence also surjective (see Corollary 3.2.31 of Gasinski \& Papageorgiou \cite[p. 319]{9}). Therefore we can find $\tilde{u}_\eta\in W^{1,p}_0(\Omega)\backslash\{0\}$ such that
	\begin{eqnarray*}
		&& A(\tilde{u}_\eta)=\eta \\
		&\Rightarrow & \langle A(\tilde{u}_\eta),h\rangle = \eta\int_\Omega hdz\ \mbox{for all}\ h\in W^{1,p}_0(\Omega).
	\end{eqnarray*}
	
	Choosing $h=-\tilde{u}^-_\eta\in W^{1,p}_0(\Omega)$ and using Lemma \ref{lem2}, we obtain
	\begin{eqnarray*}
		&&\frac{c_1}{p-1}||D\tilde{u}^-_\eta||^p_p\leq0,\\
		& \Rightarrow & \tilde{u}_\eta\geq0,\ \tilde{u}_\eta\neq0\ \mbox{(since $\eta>0$)}.
	\end{eqnarray*}
	
	We have
	\begin{equation}\label{eq12}
		-{\rm div}\,a(D\tilde{u}_\eta(z))=\eta\ \mbox{for almost all}\ z\in\Omega,\ \tilde{u}_\eta|_{\partial\Omega}=0.
	\end{equation}
	
	Theorem 7.1 of Ladyzhenskaya \& Uraltseva \cite[p. 286]{16}, implies that $\tilde{u}_\eta\in L^\infty(\Omega)$. Then the nonlinear regularity theory of Lieberman \cite[p. 320]{18} implies that
	\begin{equation*}
		\tilde{u}_\eta\in C_+\backslash\{0\}.
	\end{equation*}
	
	From \eqref{eq12} and the nonlinear strong maximum principle of Pucci \& Serrin \cite[pp. 111,120]{31}, we have
	\begin{equation*}
		\tilde{u}_\eta\in {\rm int}\,C_+.
	\end{equation*}
	
	Moreover, this solution is unique on account of the strict monotonicity of $A(\cdot)$.
	
	Next, suppose that $0<\mu<\eta\leq1$. We have
	\begin{eqnarray*}
		&&-{\rm div}\,a(D\tilde{u}_\mu(z)) = \mu<\eta = -{\rm div}\,a(D\tilde{u}_\eta(z))\ \mbox{for almost all}\ z\in\Omega, \\
		& \Rightarrow & \tilde{u}_\eta-\tilde{u}_\mu\in {\rm int}\,C_+\ \mbox{(see Gasinski \& Papageorgiou \cite[Proposition 3.4]{11})}.
	\end{eqnarray*}
	
	Finally, let $\eta_n\rightarrow0^+$ and let $\tilde{u}_n=\tilde{u}_{\eta_n}\in {\rm int}\,C_+$. Clearly, $\{\tilde{u}_n\}_{n\geq1}\subseteq W^{1,p}_0(\Omega)$ is bounded and from Ladyzhenskaya \& Uraltseva \cite[p. 286]{16}, we have
	\begin{equation*}
		||\tilde{u}_n||_\infty\leq c_{10}\ \mbox{for some}\ c_{10}>0,\ \mbox{all}\ n\in\NN.
	\end{equation*}
	
	Then it follows from Lieberman \cite{18} that there exist $\alpha\in (0,1)$ and $c_{11}>0$ such that
	\begin{equation*}
		\tilde{u}_n\in C^{1,\alpha}_0(\overline\Omega)\ \mbox{and}\ ||\tilde{u}_n||_{C^{1,\alpha}_0(\overline\Omega)}\leq c_{11}\ \mbox{for all}\ n\in\NN.
	\end{equation*}
	
	Exploiting the compact embedding of $C^{1,\alpha}_0(\overline\Omega)$ into $C^1_0(\overline\Omega)$ and recalling that $\{\tilde{u}_n\}_{n\geq1}$ is decreasing, we conclude that
	\begin{equation*}
		\tilde{u}_n\rightarrow0\ \mbox{in}\ C^1_0(\overline\Omega)
	\end{equation*}
	(note that $u=0$ is the only solution of $A(u)=0$).
\end{proof}

Fix $\lambda>0$. On account of Proposition \ref{prop9} and hypothesis $H(\theta)(ii)$, we can find small enough $\eta_0\in(0,1]$ such that
\begin{equation}\label{eq13}
	0<\frac{\eta}{\theta(\tilde{u}_\eta)}\leq\lambda\ \mbox{for all}\ \eta\in(0,\eta_0].
\end{equation}

So, we have
\begin{equation}\label{eq14}
	-{\rm div}\,a(D\tilde{u}_\eta(z))=\eta\leq\lambda\vartheta(\tilde{u}_\eta(z))\ \mbox{for almost all}\ z\in\Omega.
\end{equation}

We consider the following purely singular nonlinear, nonhomogeneous Dirichlet problem
\begin{equation}\label{eq15}
	-{\rm div}\,a(Du(z)) = \lambda\vartheta(u(z))\ \mbox{in}\ \Omega,\ u|_{\partial\Omega}=0,\ \lambda>0.
\end{equation}

\begin{prop}\label{prop10}
	If hypotheses $H(a)(i)(ii)(iii)$ hold and $\lambda>0$, then problem \eqref{eq15} admits a unique positive solution $\overline{u}_\lambda\in {\rm int}\,C_+$.
\end{prop}

\begin{proof}
	Let $\eta\in(0,\eta_0]$ and let $\tilde{u}_\eta\in {\rm int}\,C_+$ be the unique positive solution of \eqref{eq11} (see Proposition \ref{prop9} and \eqref{eq13}). We introduce the Carath\'eodory function $k_\lambda(z,x)$ defined by
	\begin{equation}\label{eq16}
		k_\lambda(z,x)=\left\{
			\begin{array}{ll}
				\lambda\vartheta(\tilde{u}_\eta(z)) &\mbox{if}\ x\leq\tilde{u}_\eta(z)\\
				\lambda\vartheta(x) &\mbox{if}\ \tilde{u}_\eta(z)<x.
			\end{array}
		\right.
	\end{equation}
	
	We set $K_\lambda(z,x)=\int^x_0k_\lambda(z,s)ds$ and consider the functional $\Psi_\lambda: W^{1,p}_0(\Omega)\rightarrow\RR$ defined by
	\begin{equation*}
		\Psi_\lambda(u)=\int_\Omega G(Du)dz-\int_\Omega K_\lambda(z,u)dz\ \mbox{for all}\ u\in W^{1,p}_0(\Omega).
	\end{equation*}
	
	Proposition \ref{prop5} implies that $\Psi_\lambda\in C^1(W^{1,p}_0(\Omega),\RR)$. Hypothesis $H(a)(iv)$ implies that if $c_{12}>c^*$, then we can find $\delta>0$ such that
	\begin{equation}\label{eq17}
		G(y)\leq c_{12}|y|^q\ \mbox{for all}\ |y|\leq\delta.
	\end{equation}
	
	Recall that $\tilde{u}_\eta,\hat{u}_1(q)\in {\rm int}\,C_+$. So, we can find small enough $t\in(0,1)$ such that
	\begin{equation}\label{eq18}
		t|D\hat{u}_1(q)(z)|\leq\delta\ \mbox{for all}\ z\in\overline\Omega\ \mbox{and}\ t\hat{u}_1(q)\leq\tilde{u}_\eta\ \mbox{(see Proposition \ref{prop8})}.
	\end{equation}
	
	Then we have
	\begin{eqnarray}
		&\Psi_\lambda(t\hat{u}_1(q)) &=\int_\Omega G(D(t\hat{u}_1(q)))dz - \int_\Omega K_\lambda(z,t\hat{u}_1(q)dz) \nonumber \\
		&&\leq c_{12}t^q\hat\lambda_1(q)-\lambda t\int_\Omega\vartheta(\tilde{u}_\eta)\hat{u}_1(q)dz \nonumber \\
		&&\mbox{(see \eqref{eq16}, \eqref{eq17}, \eqref{eq18} and recall that $||\hat{u}_1(q)||_q=1$)} \nonumber \\
		&&\leq c_{12} t^q\hat\lambda_1(q)-\lambda tc_{13}\ \mbox{for some}\ c_{13}>0. \label{eq19}
	\end{eqnarray}
	
	Since $q>1$, choosing $t\in(0,1)$ even smaller if necessary, we infer from \eqref{eq19} that
	\begin{equation}\label{eq20}
		\Psi_\lambda(t\hat{u}_1(q))<0.
	\end{equation}
	
	Let $u\in W^{1,p}_0(\Omega)$. We have
	\begin{eqnarray}
		&\Psi_\lambda(u) &\geq\frac{c_1}{p(p-1)}||Du||^p_p - \int_{\{u\leq\tilde{u}_\eta\}}\lambda\vartheta(\tilde{u}_\eta)udz - \int_{\{\tilde{u}_\eta<u\}}\lambda[\theta(u)-\theta(\tilde{u}_\eta)]dz \nonumber \\
		&& \mbox{(see Corollary \ref{cor3} and \eqref{eq16})} \nonumber \\
		&& \geq \frac{c_1}{p(p-1)}||Du||^p_p -\int_{\{u\leq\tilde{u}_\eta\}}\lambda\vartheta(\tilde{u}_\eta)udz - \int_{\{\tilde{u}_\eta<u\}}\int^{u(z)}_{\tilde{u}_\eta(z)}\lambda\vartheta(s)dsdz \nonumber \\
		&& \geq \frac{c_1}{p(p-1)}||Du||^p_p - \int_\Omega\lambda\vartheta(\tilde{u}_\eta)udz\ \mbox{(see hypothesis $H(\theta)(ii)$)} \nonumber \\
		&& \geq \frac{c_1}{p(p-1)}||u||^p - \lambda c_{14}||u||\ \mbox{for some}\ c_{14}>0. \label{eq21}
	\end{eqnarray}
	
	We choose $\rho_0>0$ such that
	\begin{equation*}
		\frac{c_1}{p(p-1)}\rho^{p-1}_0>\lambda c_{14}.
	\end{equation*}
	
	Then from \eqref{eq21}, we have
	\begin{equation}\label{eq22}
		\Psi_\lambda|_{\partial B_{\rho_0}}>0,
	\end{equation}
	where $B_{\rho_0}=\{u\in W^{1,p}_0(\Omega):||u||<\rho_0\}$.
	
	Evidently, $\Psi_\lambda(\cdot)$ is sequentially weakly lower semicontinuous, while $\overline{B}_{\rho_0}\subseteq W^{1,p}_0(\Omega)$ is weakly compact. So, we can find $\overline{u}_\lambda\in\overline{B}_{\rho_0}$ such that
	\begin{equation}\label{eq23}
		\Psi_\lambda(\overline{u}_\lambda)=\inf\{\Psi_\lambda(u):u\in\overline{B}_{\rho_0}\}.
	\end{equation}
	
	On account of \eqref{eq20}, we have
	\begin{eqnarray}
		&&\Psi_\lambda(\overline{u}_\lambda)<0=\Psi_\lambda(0), \nonumber\\
		&\Rightarrow & \overline{u}_\lambda\neq0\ \mbox{and}\ ||\overline{u}_\lambda||<\rho_0\ \mbox{(see \eqref{eq22})}. \label{eq24}
	\end{eqnarray}
	
	Then \eqref{eq23} and \eqref{eq24} imply that
	\begin{eqnarray}
		&&\Psi'_\lambda(\overline{u}_\lambda)=0, \nonumber\\
		&\Rightarrow & \langle A(\overline{u}_\lambda),h\rangle=\int_\Omega k_\lambda(z,\overline{u}_\lambda)hdz\ \mbox{for all}\ h\in W^{1,p}_0(\Omega). \label{eq25}
	\end{eqnarray}
	
	In \eqref{eq25} we choose $h=(\tilde{u}_\eta-\overline{u}_\lambda)^+\in W^{1,p}_0(\Omega)$. Then
	\begin{eqnarray}
		& \langle A(\overline{u}_\lambda),(\tilde{u}_\eta-\overline{u}_\lambda)^+\rangle &= \int_\Omega\lambda\vartheta(\tilde{u}_\eta)(\tilde{u}_\eta-\overline{u}_\lambda)^+dz\ \mbox{(see \eqref{eq16})} \nonumber \\
		&& \geq\langle A(\tilde{u}_\eta),(\tilde{u}_\eta-\overline{u}_\lambda)^+\rangle \ \mbox{(see \eqref{eq14})},\nonumber \\
		&& \Rightarrow\tilde{u}_\eta\leq\overline{u}_\lambda\ \mbox{(see Lemma \ref{lem4})}. \label{eq26}
	\end{eqnarray}
	
	It follows from \eqref{eq16}, \eqref{eq25}, \eqref{eq26} that
	\begin{equation*}
		-{\rm div}\,a(D\overline{u}_\lambda(z))=\lambda\vartheta(\overline{u}_\lambda(z))\ \mbox{for almost all}\ z\in\Omega,\ \overline{u}_\lambda|_{\partial\Omega}\neq0.
	\end{equation*}
	
	As before, the nonlinear regularity theory and \eqref{eq26} imply that
	\begin{equation*}
		\overline{u}_\lambda\in {\rm int}\,C_+.
	\end{equation*}
	
	Now we show the uniqueness of this positive solution of problem \eqref{eq15}. So, suppose $\overline{v}_\lambda\in W^{1,p}_0(\Omega)$ is another such positive solution. Again, we have $\overline{v}_\lambda\in {\rm int}\,C_+$. Then
	\begin{eqnarray*}
		&& 0\leq\langle A(\overline{u}_\lambda)-A(\overline{v}_\lambda),(\overline{u}_\lambda-\overline{v}_\lambda)^+\rangle = \lambda\int_\Omega[\vartheta(\overline{u}_\lambda)-\vartheta(\overline{v}_\lambda)](\overline{u}_\lambda-\overline{v}_\lambda)^+dz\leq0\\
		&& \mbox{(see hypothesis $H(\theta)(ii)$)} \\
		&\Rightarrow & \overline{u}_\lambda = \overline{v}_\lambda\ \mbox{(see Lemma \ref{lem4})}.
	\end{eqnarray*}
	
	Therefore $\overline{u}_\lambda\in {\rm int}\,C_+$ is the unique positive solution.
\end{proof}

\begin{remark}
	Existence results for purely singular problems can also be found in the papers of Lair \& Shaker \cite{17}, Sun, Wu \& Long \cite{32} (semilinear problems with $a(y)=y$) and by Giacomoni, Schindler \& Taka\v{c} \cite{13}, Papageorgiou \& Smyrlis \cite{28} (nonlinear problems with $a(y)=|y|^{p-2}y$). In all the aforementioned works the singular term has the particular form $u^{-\beta}$ with $\beta\in(0,1)$. In Lair \& Shaker \cite{17} the proof is based on a monotone iterative process which makes use of the linearity of the differential operator. In Sun, Wu \& Long \cite{32} and Giacomoni, Schindler \& Taka\v{c} \cite{13}, the approach is the same and is based on an argument which shows that the energy functional is G\^ateaux differentiable at points in ${\rm int}\,C_+$. Papageorgiou  \& Smyrlis \cite{28} have a proof based on approximate equations which eliminate the singularity at $x=0$. Our approach in this paper is different from the ones in the aforementioned papers and we believe that it is more straightforward and accomodates easily both the more general differential operator and the more general singular term.
\end{remark}

Now we are ready to prove the nonemptiness of the set $\mathcal{L}$.

\begin{prop}\label{prop11}
	If hypotheses $H(a),H(\theta),H(f)$ hold, then $\mathcal{L}\neq\emptyset$ and $S_\lambda\subseteq {\rm int}\,C_+$.
\end{prop}

\begin{proof}
	Let $\overline{u}_\lambda\in {\rm int}\,C_+$ be the unique positive solution of problem \eqref{eq15} produced in Proposition \ref{prop10}. We introduce the following truncation of the reaction in problem \eqref{eqp} $(\lambda>0)$
	\begin{equation}\label{eq27}
		\tau_\lambda(z,x)=\left\{
			\begin{array}{ll}
				\lambda\vartheta(\overline{u}_\lambda(z))+f(z,x) &\mbox{if}\ x\leq\overline{u}_\lambda(z) \\
				\lambda\vartheta(x)+f(z,x) &\mbox{if}\ \overline{u}_\lambda(z)<x.
			\end{array}
		\right.
	\end{equation}
	
	This is a Carath\'eodory function. Let $T_\lambda(z,x)=\int^x_0\tau_\lambda(z,s)ds$ and consider the functional $\varphi_\lambda: W^{1,p}_0(\Omega)\rightarrow\RR$ defined by
	\begin{equation*}
		\varphi_\lambda(u)=\int_\Omega G(Du)dz - \int_\Omega T_\lambda(z,u)dz\ \mbox{for all}\ u\in W^{1,p}_0(\Omega).
	\end{equation*}
	
	Using Proposition \ref{prop5}, we see that $\varphi_\lambda\in C^1(W^{1,p}_0(\Omega),\RR)$. We have
	\begin{eqnarray}
		&\varphi_\lambda(u) &\geq\frac{c_1}{p(p-1)}||Du||^p_p - \int_{\{u\leq\overline{u}_\lambda\}}\lambda\vartheta(\overline{u}_\lambda)udz - \int_{\{\overline{u}_\lambda<u\}}\lambda\int^{u(z)}_{\overline{u}_\lambda(z)}\vartheta(s)dsdz- \nonumber \\
		&&\int_\Omega F(z,u)dz\ \mbox{for all}\ u\in W^{1,p}_0(\Omega)\ \mbox{(see Corollary \ref{cor3} and \eqref{eq27})} \nonumber\\
		&&\geq\frac{c_1}{p(p-1)}||Du||^p_p - \int_\Omega\lambda\vartheta(\overline{u}_\lambda)udz - \int_\Omega F(z,u)dz \nonumber\\
		&&\mbox{for all}\ u\in W^{1,p}_0(\Omega)\ \mbox{(see hypothesis $H(\theta)(ii)$)} \nonumber\\
		&&\geq\frac{c_1}{p(p-1)}||u||^p-\lambda c_{15}||u||-\int_\Omega F(z,u)dz \label{eq28}\\
		&&\mbox{for all}\ u\in W^{1,p}_0(\Omega),\ \mbox{and some}\ c_{15}>0. \nonumber
	\end{eqnarray}
	
	Hypotheses $H(f)(i)(iv)$ imply that given $\epsilon>0$, we can find $c_{16}=c_{16}(\epsilon)>0$ such that
	\begin{equation*}
		F(z,x)\leq \frac{\epsilon}{p}x^p+c_{16}x^r\ \mbox{for almost all}\ z\in\Omega,\ \mbox{and all}\ x\geq0.
	\end{equation*}
	
	Using this growth estimate in \eqref{eq28} and choosing $\epsilon>0$ small, we obtain
	\begin{eqnarray}\label{eq29}
		\varphi_\lambda(u)\geq c_{17}||u||^p-c_{18}||u||^r -\lambda c_{15}||u||\\
		\mbox{for all}\ u\in W^{1,p}_0(\Omega),\ \mbox{and some}\ c_{17},c_{18}>0. \nonumber
	\end{eqnarray}
	
	We first  choose $0<\rho<\left(\frac{c_{17}}{c_{18}}\right)^\frac{1}{r-p}$. Let $\tilde\xi=c_{17}\rho^p-c_{18}\rho^r>0$. We choose $\lambda_0>0$ small such that
	\begin{equation*}
		\lambda c_{15}\rho<\tilde\xi\ \mbox{for all}\ \lambda\in(0,\lambda_0).
	\end{equation*}
	
	Then from \eqref{eq29} we have
	\begin{equation}\label{eq30}
		\varphi_\lambda|_{\partial B_\rho}>0\ \mbox{for all}\ \lambda\in(0,\lambda_0)
	\end{equation}
	(recall that $B_\rho=\{u\in W^{1,p}_0(\Omega):||u||<\rho\}$).
	
	Let $\delta>0$ be as in \eqref{eq17}. We choose $t\in(0,1)$ so small that
	\begin{equation}\label{eq31}
		t|D\hat{u}_1(q)(z)|\leq\delta\ \mbox{for all}\ z\in\overline\Omega\ \mbox{and}\ t\hat{u}_1(q)\leq\overline{u}_\lambda
	\end{equation}
	(recall that $\hat{u}_1(q),\overline{u}_\lambda\in {\rm int}\,C_+$ and use Proposition \ref{prop8}). We have
	\begin{eqnarray*}
		&\varphi_\lambda(t\hat{u}_1(q)) &\leq t^q\hat{\lambda}_1(q) - t\int_\Omega\lambda\vartheta(\overline{u}_\lambda)\hat{u}_1(q)dz - \int_\Omega F(z,t\hat{u}_1(q))dz\\
		&&\mbox{(see \eqref{eq17}, \eqref{eq27}, \eqref{eq31} and recall that $||\hat{u}_1(q)||_q=1$)}\\
		&&\leq t^q\hat\lambda_1(q)-\lambda tc_{19}\ \mbox{for some}\ c_{19}>0\ \mbox{(recall that $F\geq0$)}.
	\end{eqnarray*}
	
	Since $q>1$, choosing $t\in(0,1)$ even smaller if necessary, we can infer that
	\begin{equation}\label{eq32}
		\varphi_\lambda(t\hat{u}_1(q))<0=\varphi_\lambda(0).
	\end{equation}
	
	The functional $\varphi_\lambda$ is sequentially weakly lower semicontinuous, whereas $\overline{B}_\rho\subseteq W^{1,p}_0(\Omega)$ is $w$-compact. So, we can find $u_\lambda\in W^{1,p}_0(\Omega)$ such that
	\begin{equation*}
		\varphi_\lambda(u_\lambda)=\inf\{\varphi_\lambda(u):u\in W^{1,p}_0(\Omega)\}.
	\end{equation*}
	
	From \eqref{eq30} and \eqref{eq32} it follows that
	\begin{eqnarray}
		&& u_\lambda\neq0\ \mbox{and}\ ||u_\lambda||<\rho, \nonumber \\
		&\Rightarrow & \varphi'_\lambda(u_\lambda)=0, \nonumber \\
		&\Rightarrow & \langle A(u_\lambda),h\rangle=\int_\Omega \tau_\lambda(z,u_\lambda)hdz\ \mbox{for all}\ h\in W^{1,p}_0(\Omega). \label{eq33}
	\end{eqnarray}
	
	In \eqref{eq33} we choose $h=(\overline{u}_\lambda-u)^+\in W^{1,p}_0(\Omega)$. Then
	\begin{eqnarray}
		&\langle A(u_\lambda),(\overline{u}_\lambda-u_\lambda)^+\rangle &=\int_\Omega[\lambda\vartheta(\overline{u}_\lambda)+f(z,u_\lambda)](\overline{u}_\lambda-u_\lambda)^+dz  \mbox{(see (\eqref{eq27}))}\nonumber\\
		&&\geq\int_\Omega\lambda\vartheta(\overline{u}_\lambda)(\overline{u}_\lambda-u_\lambda)^+dz\ \mbox{(since $f\geq0$)} \nonumber \\
		&&=\langle A(\overline{u}_\lambda),(\overline{u}_\lambda-u_\lambda)^+\rangle, \nonumber \\
		&& \Rightarrow\overline{u}_\lambda\leq u_\lambda\ \mbox{(see Lemma \ref{lem4})}. \label{eq34}
	\end{eqnarray}
	
	Then it follows from \eqref{eq27}, \eqref{eq33} and \eqref{eq34} that
	\begin{eqnarray*}
		&&-{\rm div}\,a(Du_\lambda(z))=\lambda\vartheta(u_\lambda(z)) + f(z,u_\lambda(z))\ \mbox{for almost all}\ z\in\Omega, \ u_\lambda|_{\partial\Omega}=0,\\
		&\Rightarrow & u_\lambda\in {\rm int}\,C_+\ \mbox{(by the nonlinear regularity theory and \eqref{eq34})}.
	\end{eqnarray*}
	
	So, we conclude that
	\begin{equation*}
		(0,\lambda_0)\subseteq\mathcal{L}\neq\emptyset\ \mbox{and}\ S_\lambda\subseteq {\rm int}\,C_+.
	\end{equation*}
The proof is now complete.
\end{proof}

\begin{prop}\label{prop12}
	If hypotheses $H(a),H(\vartheta),H(f)$ hold, then $\lambda^*<+\infty$.
\end{prop}
\begin{proof}
	On account of hypotheses $H(f)(ii),(iii)$, we have
	$$\lim\limits_{x\rightarrow+\infty}\frac{f(z,x)}{x^{p-1}}=+\infty\ \mbox{uniformly for almost all}\ z\in\Omega.$$

Therefore we can find $M>1$ such that
$$f(z,x)\geq x^{p-1}\ \mbox{for almost all}\ z\in\Omega\ \mbox{and all}\ x\geq M.$$

Let $\tilde{\lambda}=\frac{M^{p-1}}{\vartheta(M)}$. Then
\begin{eqnarray}\label{eq35}
	&&\tilde{\lambda}\vartheta(x)\geq M^{p-1}\geq x^{p-1}\ \mbox{for a.a.}\ x\in\left(0,M\right]\ (\mbox{see hypothesis}\ H(\vartheta)(ii)),\nonumber\\
	&\Rightarrow&\tilde{\lambda}\vartheta(x)+f(z,x)\geq x^{p-1}\ \mbox{for a.a.}\ z\in\Omega,\ \mbox{all}\ x\geq 0\ (\mbox{recall that}\ \vartheta\geq 0,f\geq 0).
\end{eqnarray}

Consider $\lambda>\tilde{\lambda}$ and suppose that $\lambda\in\mathcal{L}$. Then we can find $u_\lambda\in S_\lambda\subseteq {\rm int}\, C_+$ (see Proposition \ref{prop11}). Let $\Omega_0\subseteq\Omega$ be a strict open subset of $\Omega$ with $C^2$-boundary $\partial\Omega_0$. We set
$$m_0=\min\limits_{\overline{\Omega}_0}u_\lambda>0.$$

Let $\rho=||u_\lambda||_\infty$ and let $\hat{\xi}_{\rho}>0$ be as postulated by hypothesis $H(f)(v)$. For $\epsilon>0$, we set $m^{\epsilon}_0=m_0+\epsilon$. We have
\begin{eqnarray*}
	&&-{\rm div}\, a(Dm^{\epsilon}_0)-\lambda\vartheta(m^{\epsilon}_0)+\hat{\xi}_\rho(m^{\epsilon}_0)^{p-1}\\
	&\leq&\hat{\xi}_\rho m^{p-1}_0-\lambda\vartheta(m_0)+\chi(\epsilon)\ \mbox{with}\ \chi(\epsilon)\rightarrow 0^+\ \mbox{as}\ \epsilon\rightarrow 0^+\ (\mbox{see hypotheses}\ H(\vartheta))\\
	&<&(1+\hat{\xi}_\rho)m^{p-1}_0-\lambda\vartheta(m_0)+\chi(\epsilon)\\
	&\leq&f(z,m_0)+\hat{\xi}_\rho m^{p-1}_0\ \mbox{for all small enough}\ \epsilon>0\ (\mbox{recall that}\ \lambda>\tilde{\lambda})\\
	&\leq&f(z,u_\lambda(z))+\hat{\xi}_\rho u_\lambda(z)^{p-1}\ (\mbox{see hypothesis}\ H(f)(v))\\
	&=&-{\rm div}\,a (Du_\lambda(z))-\lambda\vartheta(u_\lambda(z))+\hat{\xi}_\rho u_\lambda(z)^{p-1}\ \mbox{for almost all}\ z\in\Omega_0.
\end{eqnarray*}

Invoking Proposition \ref{prop6}, we have
$$u_\lambda-m^{\epsilon}_\lambda\in {\rm int}\,\hat{C}_+(\Omega_0)\ \mbox{for small enough}\ \epsilon>0,$$
which contradicts the fact that $m_\lambda=\min\limits_{\overline{\Omega}_0}u_\lambda.$

Therefore $\lambda\notin\mathcal{L}$ and so $\lambda^*\leq \tilde{\lambda}<+\infty$.
\end{proof}

Let $\bar{u}_\lambda\in {\rm int}\, C_+$ be the unique solution of problem (\ref{eqp}) produced in Proposition \ref{prop10}. Next, we show that $\bar{u}_\lambda$ is a lower bound for the elements of $S_\lambda$.
\begin{prop}\label{prop13}
	If hypotheses $H(a),H(\vartheta),H(f)$ hold and $\lambda\in\mathcal{L}$, then $\bar{u}_\lambda\leq u$ for all $u\in S_\lambda\subseteq {\rm int}\, C_+.$
\end{prop}
\begin{proof}
	Since $\lambda\in\mathcal{L}$ we have $\emptyset\neq S_\lambda\subseteq {\rm int}\, C_+$. Let $u\in S_\lambda$. Then
	\begin{equation}\label{eq36}
		\left\langle A(u),h\right\rangle=\int_\Omega[\lambda\vartheta(u)+f(z,u)]hdz\ \mbox{for all}\ h\in W^{1,p}_0(\Omega).
	\end{equation}
	
	In (\ref{eq36}) we choose $h=(\bar{u}_\lambda-u)^+\in W^{1,p}_0(\Omega)$. We have
	\begin{eqnarray*}
		&&\left\langle A(\bar{u}_\lambda),(\bar{u}_\lambda-u)^+\right\rangle\\
		&=&\int_\Omega[\lambda\vartheta(u)+f(z,u)](\bar{u}_\lambda-u)^+dz\\
		&\geq&\int_\Omega\lambda\vartheta(u)(\bar{u}_\lambda-u)^+dz\ (\mbox{since}\ f\geq 0)\\
		&\geq&\int_{\Omega}\lambda\vartheta(\bar{u}_\lambda)(\bar{u}_\lambda-u)^+dz\ (\mbox{see hypothesis}\ H(\vartheta)(ii))\\
		&=&\left\langle A(\bar{u}_\lambda),(\bar{u}_\lambda-u)^+\right\rangle,\\
		&\Rightarrow&\bar{u}_\lambda\leq u\ (\mbox{see Lemma \ref{lem4}}).
	\end{eqnarray*}
The proof is now complete.
\end{proof}

We also show that the map $\lambda\mapsto\bar{u}_\lambda$ from $(0,+\infty)$ into $C^1_0(\overline{\Omega})$ is nondecreasing, that is
$$0<\mu<\lambda\Rightarrow\bar{u}_\lambda-\bar{u}_\mu\in C_+.$$
\begin{prop}\label{prop14}
	If hypotheses $H(a),H(\vartheta),H(f)$ hold, then the map $\lambda\mapsto\bar{u}_\lambda$ from $(0,+\infty)$ into $C^1_0(\overline{\Omega})$ is nondecreasing.
\end{prop}
\begin{proof}
	Let $0<\mu<\lambda$. We have
	\begin{equation}\label{eq37}
		\left\langle A(\bar{u}_\mu),h\right\rangle=\int_\Omega\mu\vartheta(\bar{u}_\mu)hdz\ \mbox{for all}\ h\in W^{1,p}_0(\Omega).
	\end{equation}
	
	In (\ref{eq37}) we choose $h=(\bar{u}_\mu-\bar{u}_\lambda)^+\in W^{1,p}_0(\Omega)$. Then
	\begin{eqnarray*}
		\left\langle A(\bar{u}_\mu),(\bar{u}_\mu-\bar{u}_\lambda)^+\right\rangle&=&\int_\Omega\mu\vartheta(\bar{u}_\mu)(\bar{u}_\mu-\bar{u}_\lambda)^+dz\\
		&\leq&\int_\Omega\lambda\vartheta(\bar{u}_\lambda)(\bar{u}_\mu-\bar{u}_\lambda)^+dz\\
		&&(\mbox{since}\ \mu<\lambda\ \mbox{and see hypothesis}\ H(\vartheta)(ii))\\
		&=&\left\langle A(\bar{u}_\lambda),(\bar{u}_\mu-\bar{u}_\lambda)^+\right\rangle,\\
		\Rightarrow\bar{u}_\mu\leq\bar{u}_\lambda&&\mbox{(see Lemma \ref{lem4})}.
	\end{eqnarray*}
	
	Clearly, we also have $\bar{u}_\lambda\neq\bar{u}_\mu$.
\end{proof}

Now we can show that $\mathcal{L}$ is an interval.
\begin{prop}\label{prop15}
	If hypotheses $H(a),H(\vartheta),H(f)$ hold, $\lambda\in\mathcal{L}$ and $u_\lambda\in S_\lambda\subseteq {\rm int}\, C_+$, then $\mu\in\mathcal{L}$ and we can find $u_\mu\in S_\mu\subseteq {\rm int}\, C_+$ such that
	$$u_\lambda-u_\mu\in {\rm int}\, C_+.$$
\end{prop}
\begin{proof}
	On account of Propositions \ref{prop13} and \ref{prop14}, we have
	$$\bar{u}_\mu\leq\bar{u}_\lambda\leq u_\lambda.$$
	
	We introduce the Carath\'eodory function $\gamma_\mu(z,x)$ defined by
	\begin{eqnarray}\label{eq38}
		&&\gamma_\mu(z,x)=\left\{\begin{array}{ll}
			\mu\vartheta(\bar{u}_\mu(z))+f(z,\bar{u}_\mu(z))&\mbox{if}\ x<\bar{u}_\mu(z)\\
			\mu\vartheta(x)+f(z,x)&\mbox{if}\ \bar{u}_\mu(z)\leq x\leq u_\lambda(z)\\
			\mu\vartheta(u_\lambda(z))+f(z,u_\lambda(z))&\mbox{if}\ u_\lambda(z)<x.
		\end{array}\right.
	\end{eqnarray}
	
	Let $\Gamma_\mu(z,x)=\int^x_0\gamma_\mu(z,s)ds$ and consider the $C^1$-functional $\tilde{\varphi}_\mu:W^{1,p}_0(\Omega)\rightarrow\RR$ defined by
	$$\tilde{\varphi}_\mu(u)=\int_\Omega G(Du)dz-\int_\Omega \Gamma_\mu(z,u)dz\ \mbox{for all}\ u\in W^{1,p}_0(\Omega).$$
	
	Corollary \ref{cor3} and (\ref{eq38}) imply that $\tilde{\varphi}_\mu(\cdot)$ is coercive. Also it is sequentially weakly lower semicontinuous. So, we can find $u_\mu\in W^{1,p}_0(\Omega)$ such that
	\begin{eqnarray}\label{eq39}
		&&\tilde{\varphi}_\mu(u_\mu)=\inf\{\tilde{\varphi}_\mu(u):u\in W^{1,p}_0(\Omega)\},\nonumber\\
		&\Rightarrow&\tilde{\varphi}'_\mu(u_\mu)=0,\nonumber\\
		&\Rightarrow&\left\langle A(u_\mu),h\right\rangle=\int_\Omega\gamma_\mu(z,u_\mu)hdz\ \mbox{for all}\ h\in W^{1,p}_0(\Omega).
	\end{eqnarray}
	
	In (\ref{eq39}) we first choose $h=(\bar{u}_\mu-u_\mu)^+\in W^{1,p}_0(\Omega)$. Then
	\begin{eqnarray*}
		\left\langle A(u_\mu),(\bar{u}_\mu-u_\mu)^+\right\rangle&=&\int_\Omega[\mu\vartheta(\bar{u}_\mu)+f(z,\bar{u}_\mu)](\bar{u}_\mu-u_\mu)^+dz\ (\mbox{see (\ref{eq38})})\\
		&\geq&\int_\Omega\mu\vartheta(\bar{u}_\mu)(\bar{u}_\mu-u_\mu)^+dz\ (\mbox{since}\ f\geq 0)\\
		&=&\left\langle A(\bar{u}_\mu),(\bar{u}_\mu-u_\mu)^+\right\rangle,\\
		\Rightarrow\bar{u}_\mu\leq u_\mu. &&
	\end{eqnarray*}
	
	Next, in (\ref{eq39}) we choose $h=(u_\mu-u_\lambda)^+\in W^{1,p}_0(\Omega)$. Then
	\begin{eqnarray*}
		\left\langle A(u_\mu),(u_\mu-u_\lambda)^+\right\rangle&=&\int_\Omega[\mu\vartheta(u_\lambda)+f(z,u_\lambda)](u_\mu-u_\lambda)^+dz\ (\mbox{see (\ref{eq38})})\\
		&\leq&\int_\Omega[\lambda\vartheta(u_\vartheta)+f(z,u_\lambda)](u_\mu-u_\lambda)^+dz\ (\mbox{since}\ \mu<\lambda)\\
		&=&\left\langle A(u_\lambda),(u_\mu-u_\lambda)^+\right\rangle,\\
		\Rightarrow u_\mu\leq u_\lambda.&&
	\end{eqnarray*}
	
	So, we have proved that
	\begin{equation}\label{eq40}
		u_\mu\in[\bar{u}_\lambda,u_\lambda].
	\end{equation}
	
	From (\ref{eq38}), (\ref{eq39}), (\ref{eq40}) it follows that
	$$u_\mu\in S_\mu\subseteq {\rm int}\, C_+\ \mbox{and so}\ \mu\in\mathcal{L}.$$
	
	Now, let $\rho=||u_\lambda||_\infty$ and let $\hat{\xi}_\rho>0$ be as postulated by hypothesis $H(f)(v)$. We have
	\begin{eqnarray}\label{eq41}
		&&-{\rm div}\,a(Du_\mu(z))-\lambda\vartheta(u_\mu(z))+\hat{\xi}_\rho u_\mu(z)^{p-1}\nonumber\\
		&=&(\mu-\lambda)\,\vartheta(u_\mu(z))+f(z,u_\mu(z))+\hat{\xi}_\rho u_\mu(z)^{p-1}\nonumber\\
		&\leq&f(z,u_\lambda(z))+\hat{\xi}_\rho u_\mu(z)^{p-1}\ (\mbox{since}\ \lambda>\mu\ \mbox{and using hypothesis}\ H(f)(v))\nonumber\\
		&=&-{\rm div}\,a(Du_\lambda(z))-\lambda\vartheta(u_\lambda(z))+\hat{\xi}_\rho u_\lambda(z)^{p-1}\ \mbox{for almost all}\ z\in\Omega\\
		&&(\mbox{recall that}\ u_\lambda\in S_\lambda).\nonumber
	\end{eqnarray}
	
	Note that
	$$0\prec(\lambda-\mu)\,\vartheta(u_\mu(\cdot)).$$
	
	Therefore from (\ref{eq41}) and Proposition \ref{prop7}, we conclude that
	$$u_\lambda-u_\mu\in {\rm int}\, C_+.$$
The proof is now complete.
\end{proof}

Next, we show that the critical parameter $\lambda^*>0$ is admissible (that is, $\lambda^*\in\mathcal{L}$).
\begin{prop}\label{prop16}
	If hypotheses $H(a),H(\vartheta),H(f)$ hold, then $\lambda^*\in\mathcal{L}$.
\end{prop}
\begin{proof}
	Let $\{\lambda_n\}_{n\geq 1}\subseteq\mathcal{L}$ be such that $\lambda_n\uparrow\lambda^*$. We consider the Carath\'eodory functions $j_n:\Omega\times\RR\rightarrow\RR$ defined by
	\begin{equation}\label{eq42}
		j_n(z,x)=\left\{\begin{array}{ll}
			\lambda_n\vartheta(\bar{u}_{\lambda_1}(z))+f(z,\bar{u}_{\lambda_1}(z))&\mbox{if}\ x\leq\bar{u}_{\lambda_1}(z)\\
			\lambda_n\vartheta(x)+f(z,x)&\mbox{if}\ \bar{u}_{\lambda_1}(z)<x.
		\end{array}\right.
	\end{equation}
	
	We set $J_n(z,x)=\int^x_0j_n(z,s)ds$ and consider the $C^1$-functional $\varphi_{\lambda_n}:W^{1,p}_0(\Omega)\rightarrow\RR$ defined by
	$$\varphi_{\lambda_n}(u)=\int_\Omega G(Du)dz-\int_\Omega J_n(z,u)dz\ \mbox{for all}\ u\in W^{1,p}_0(\Omega)$$
	(see Proposition \ref{prop5}). Since $\lambda_n\in\mathcal{L}$, we can find $\hat{u}_n\in S_{\lambda_n}\subseteq {\rm int}\, C_+$ for all $n\in\NN$. Let $\bar{\gamma}_{\lambda_n}(z,x)$ be the Carath\'eodory function defined by
	\begin{equation}\label{eq43}
		\bar{\gamma}_{\lambda_n}(z,x)=\left\{\begin{array}{ll}
			\lambda_n\vartheta(\bar{u}_{\lambda_1}(z))+f(z,\bar{u}_{\lambda_1}(z))&\mbox{if}\ x<\bar{u}_{\lambda_1}(z)\\
			\lambda_n\vartheta(x)+f(z,x)&\mbox{if}\ \bar{u}_{\lambda_1}(z)\leq x\leq\hat{u}_{n+1}(z)\\
			\lambda_n\vartheta(\hat{u}_{n+1}(z))+f(z,\hat{u}_{n+1}(z))&\mbox{if}\ \hat{u}_{n+1}(z)<x.
		\end{array}\right.
	\end{equation}
	
	We set $\bar{\Gamma}_{\lambda_n}(z,x)=\int^x_0\bar{\gamma}_{\lambda_n}(z,s)ds$ and then using Proposition \ref{prop5}, we introduce the $C^1$-functional $\bar{\varphi}_{\lambda_n}:W^{1,p}_0(\Omega)\rightarrow\RR$ defined by
	$$\bar{\varphi}_{\lambda_n}(u)=\int_\Omega G(Du)dz-\int_\Omega\bar{\Gamma}_{\lambda_n}(z,u)dz\ \mbox{for all}\ u\in W^{1,p}_0(\Omega).$$
	
	It is clear from Corollary \ref{cor3} and (\ref{eq43}) that $\bar{\varphi}_{\lambda_n}(\cdot)$ is coercive. Also, it is sequentially weakly lower semicontinuous. So, we can find $u_n\in W^{1,p}_0(\Omega)$ such that
	\begin{equation}\label{eq44}
		\bar{\varphi}_{\lambda_n}(u_n)=\inf\{\bar{\varphi}_{\lambda_n}(u):u\in W^{1,p}_0(\Omega)\}.
	\end{equation}
	
	We have
	\begin{eqnarray*}
		\bar{\varphi}_{\lambda_n}(u_n)&\leq&\bar{\varphi}_{\lambda_n}(\bar{u}_{\lambda_1})\\
		&\leq&\int_\Omega G(D\bar{u}_{\lambda_1})dz-\lambda_n\int_\Omega\vartheta(\bar{u}_{\lambda_1})dz\\
		&&(\mbox{see (\ref{eq43}) and recall that}\ f\geq 0)\\
		&\leq&\left\langle A(\bar{u}_{\lambda_1}),\bar{u}_{\lambda_1}\right\rangle-\lambda_1\int_\Omega\vartheta(\bar{u}_{\lambda_1})\bar{u}_{\lambda_1}dz\\
		&&(\mbox{see (\ref{eq2}) and recall that}\ 0<\lambda_1\leq\lambda_n,\ n\in\NN)\\
		&=&0\ (\mbox{since}\ \bar{u}_{\lambda_1}\in {\rm int}\, C_+\ \mbox{solves (\ref{eq15})}).
	\end{eqnarray*}
	
	From (\ref{eq42}) and (\ref{eq43}), we see that
	\begin{eqnarray}\label{eq45}
		&&\varphi_{\lambda_n}=\bar{\varphi}_{\lambda_n}\ \mbox{on}\ [\bar{u}_{\lambda_1},\hat{u}_{n+1}]\ \mbox{for all}\ n\in\NN,\nonumber\\
		&\Rightarrow&\varphi_{\lambda_n}(u_n)\leq 0\ \mbox{for all}\ n\in\NN.
	\end{eqnarray}
	
	Since $u_n\in K_{\bar{\varphi}_{\lambda_n}}$ (see (\ref{eq44})), as in the proof of Proposition \ref{prop15}, using (\ref{eq43}) we can show that $u\in[\bar{u}_\lambda,\hat{u}_{n+1}]\cap {\rm int}\, C_+$ for all $n\in\NN$. Therefore for every $n\in\NN$, we have
	\begin{eqnarray}\label{eq46}
		&&\left\langle A(u_n),h\right\rangle=\int_\Omega j_n(z,u_n)hdz\ \mbox{for all}\ n\in\NN\ \mbox{and all}\ h\in W^{1,p}_0(\Omega).
	\end{eqnarray}
	
	In (\ref{eq46}) we choose $h=u_n\in W^{1,p}_0(\Omega)$. Then
	\begin{equation}\label{eq47}
		-\int_\Omega(a(Du_n),Du_n)_{\RR^N}dz+\int_\Omega j_n(z,u_n)u_ndz=0\ \mbox{for all}\ n\in\NN.
	\end{equation}
	
	From (\ref{eq45}) we have
	\begin{equation}\label{eq48}
		\int_\Omega pG(Du_n)dz-\int_\Omega pJ_n(z,u_n)dz\leq 0\ \mbox{for all}\ n\in\NN.
	\end{equation}
	
	We add (\ref{eq47}) and (\ref{eq48}) and obtain
	\begin{eqnarray}\label{eq49}
		&&\int_\Omega[pG(Du_n)-(a(Du_n),Du_n)_{\RR^N}]dz+\int_\Omega[j_n(z,u_n)u_n-pJ_n(z,u_n)]dz\leq 0\nonumber\\
		&&\mbox{for all}\ n\in\NN,\nonumber\\
		&\Rightarrow&\int_\Omega[j_n(z,u_n)u_n-pJ_n(z,u_n)]dz\leq 0\ \mbox{for all}\ n\in\NN\ (\mbox{see hypothesis}\ H(a)(iv)),\nonumber\\
		&\Rightarrow&\int_\Omega[(\lambda_n\vartheta(u_n)+f(z,u_n))u_n-p(\lambda_n\theta(u_n)+F(z,u_n))]dz\leq c_{20}\ \mbox{for some}\ c_{20}>0,\nonumber\\
		&&\mbox{all}\ n\in\NN\ (\mbox{see (\ref{eq42})}),\nonumber\\
		&\Rightarrow&\int_{\Omega}e_{\lambda_n}(z,u_n)dz\leq c_{20}\ \mbox{for all}\ n\in\NN.
	\end{eqnarray}
	\begin{claim}\label{cl16.1}
		$\{u_n\}_{n\geq 1}\subseteq W^{1,p}_0(\Omega)$ is bounded.
	\end{claim}
	
	We argue by contradiction. So, suppose that the claim is not true. Then by passing to a subsequence if necessary, we may assume that
	\begin{equation}\label{eq50}
		||u_n||\rightarrow+\infty.
	\end{equation}
	
	We set $y_n=\frac{u_n}{||u_n||}$ for all $ n\in\NN$. Then $||y_n||=1,\ y_n\geq 0$ for all $n\in\NN$. So, we may assume that
	\begin{equation}\label{eq51}
		y_n\stackrel{w}{\rightarrow}y\ \mbox{in}\ W^{1,p}_0(\Omega)\ \mbox{and}\ y_n\rightarrow y\ \mbox{in}\ L^p(\Omega),\ y\geq 0.
	\end{equation}
	
	First, we assume that $y\neq 0$. If $\Omega_+=\{z\in\Omega:y(z)>0\}$, then $|\Omega_+|_N>0$ where by $|\cdot|_N$ we denote the Lebesgue measure on $\RR^N$ (recall that $y\geq 0,\ y\neq 0$). We have
	$$u_n(z)\rightarrow+\infty\ \mbox{for all}\ z\in\Omega_+\ (\mbox{see (\ref{eq50})}).$$
	
	On account of hypothesis $H(f)(ii)$, we have
	$$\frac{pF(z,u_n(z))}{||u_n||^p}=\frac{pF(z,u_n(z))}{u_n(z)^p}y_n(z)^p\rightarrow+\infty\ \mbox{for almost all}\ z\in\Omega_+.$$
	
	Then from Fatou's lemma (see hypothesis $H(f)(ii)$), we have
	\begin{eqnarray}\label{eq52}
		&&\frac{1}{||u_n||^p}\int_{\Omega_+}pF(z,u_n)dz\rightarrow+\infty,\nonumber\\
		&\Rightarrow&\frac{1}{||u_n||^p}\int_\Omega pF(z,u_n)dz\rightarrow+\infty\ (\mbox{since}\ F\geq 0).
	\end{eqnarray}
	
	Also from hypotheses $H(\vartheta)(i),(ii)$ we see that
	\begin{equation}\label{eq53}
		0\leq\frac{1}{||u_n||^p}\int_\Omega\theta (u_n)dz\rightarrow 0\ \mbox{as}\ n\rightarrow\infty.
	\end{equation}
	
	So, finally we can say that
	\begin{equation}\label{eq54}
		\frac{1}{||u_n||^p}\int_\Omega[\lambda_n\theta(u_n)+pF(z,u_n)]dz\rightarrow+\infty\ \mbox{as}\ n\rightarrow\infty\ (\mbox{see (\ref{eq52}), (\ref{eq53})}).
	\end{equation}
	
	From hypothesis $H(f)(iii)$, we have for almost all $z\in\Omega$, and all $x\geq 0$
	\begin{eqnarray}\label{eq55}
		&&0\leq e_{\lambda_n}(z,x)+\hat{\beta}_{\lambda^*}(z),\nonumber\\
		&\Rightarrow&p[\lambda_n\theta(x)+F(z,x)]-\hat{\beta}_{\lambda^*}(z)\leq[\lambda_n\vartheta(x)+f(z,x)]x\,.
	\end{eqnarray}
	
	From (\ref{eq47}) and hypothesis $H(a)(iv)$, we have
	\begin{eqnarray}\label{eq56}
		&&\int_\Omega j_n(z,u_n)u_ndz\leq\int_\Omega pG(Du_n)dz\ \mbox{for all}\ n\in\NN,\nonumber\\
		&\Rightarrow&\int_\Omega p[\lambda_n\theta(u_n)+F(z,u_n)]dz\leq\int_\Omega pG(Du_n)dz+c_{21}\nonumber\\
		&&\mbox{for some}\ c_{21}>0,\ \mbox{all}\ n\in\NN\ (\mbox{see (\ref{eq42}) and (\ref{eq55})}),\nonumber\\
		&\Rightarrow&\frac{p}{||u_n||^p}\int_\Omega[\lambda_n\theta(u_n)+F(z,u_n)]dz\leq c_{22}\left(\frac{1}{||u_n||^p}+||y_n||^p\right)\leq c_{23}\\
		&&\mbox{for some}\ c_{22},c_{23}>0\ \mbox{and all}\ n\in\NN\ (\mbox{see Corollary \ref{cor3}}).\nonumber
	\end{eqnarray}
	
	We compare (\ref{eq54}) and (\ref{eq56}) and we obtain a contradiction.
	
	Now assume that $y\equiv 0$. For $\eta>0$, we define
	$$v_n=(p\eta)^{1/p}y_n\in W^{1,p}_0(\Omega)\ \mbox{for all}\ n\in\NN.$$
	
	From (\ref{eq50}) we see that we can find $n_0\in\NN$ such that
	\begin{equation}\label{eq57}
		(p\eta^{1/p})\frac{1}{||u_n||^p}\leq 1\ \mbox{for all}\ n\geq n_0.
	\end{equation}
	
	Consider the $C^1$-functional $\Im_{\lambda_n}:W^{1,p}_0(\Omega)\rightarrow\RR$ defined by
	$$\Im_{\lambda_n}(u)=\frac{c_1}{p(p-1)}||Du||^p_p-\int_\Omega J_{\lambda_n}(z,u)dz\ \mbox{for all}\ u\in W^{1,p}_0(\Omega).$$
	
	Let $t_n\in[0,1]$ (for $n\in\NN$) be such that
	\begin{equation}\label{eq58}
		\Im_{\lambda_n}(t_n,u_n)=\max\{\Im_{\lambda_n}(tu_n):0\leq t\leq 1\}.
	\end{equation}
	
	From (\ref{eq57}) and (\ref{eq58}) it follows that
	\begin{eqnarray*}
		&&\Im_{\lambda_n}(t_nu_n)\geq \Im_{\lambda(v_n)}\\
		&\geq&\frac{\eta}{p-1}||Dy_n||^p_p-\int_\Omega J_{\lambda_n}(z,v_n)dz\\
		&\geq&\frac{\eta}{p-1}-\int_\Omega J_{\lambda^*}(z,v_n)dz\ (\mbox{since}\ ||y_n||=1\ \mbox{and}\ \lambda_n\leq\lambda^*\ \mbox{for all}\ n\in\NN)\\
		&\geq&\frac{\eta}{p-1}-\int_\Omega[\lambda^*\theta(v_n)+F(z,v_n)]dz-c_{24}||v_n||\\
		&&\mbox{for some}\ c_{24}>0,\ \mbox{and all}\ n\geq n_0\ (\mbox{see (\ref{eq42})}).
	\end{eqnarray*}
	
	Since we have assumed that $y\equiv 0$, we have
	\begin{eqnarray*}
		&&\int_\Omega[\lambda^*\theta(v_n)+F(z,v_n)]dz+c_{24}||v_n||\rightarrow 0\ \mbox{as}\ n\rightarrow\infty,\\
		&\Rightarrow&\Im_{\lambda_n}(t_nu_n)\geq\frac{1}{2}\ \frac{\eta}{p-1}\ \mbox{for all}\ n\geq n_1\geq n_0.
	\end{eqnarray*}
	
	But $\eta>0$ is arbitrary. So, we can infer that
	\begin{equation}\label{eq59}
		\Im_{\lambda_n}(t_nu_n)\rightarrow+\infty\ \mbox{as}\ n\rightarrow\infty.
	\end{equation}
	
	Note that
	$$\Im_{\lambda_n}(0)=0\ \mbox{and}\ \Im_{\lambda_n}(u_n)\leq\varphi_{\lambda_n}(u_n)\leq 0\ \mbox{for all}\ n\in\NN\ (\mbox{see Corollary \ref{cor3} and (\ref{eq45})}).$$
	
	Then from (\ref{eq59}) it follows that
	\begin{eqnarray}\label{eq60}
		&&t_n\in(0,1)\ \mbox{for all}\ n\geq n_2,\nonumber\\
		&\Rightarrow&\frac{d}{dt}\Im_{\lambda_n}(tu_n)|_{t=t_n}=0\ (\mbox{see (\ref{eq58})})\nonumber\\
		&\Rightarrow&\left\langle \Im'_{\lambda_n}(t_nu_n),t_nu_n\right\rangle=0\ (\mbox{by the chain rule}),\nonumber\\
		&\Rightarrow&\frac{c_1}{p-1}||D(t_nu_n)||^p_p=\int_\Omega j_{\lambda_n}(z,t_nu_n)(t_nu_n)dz\nonumber\\
		&&\leq\int_\Omega[\lambda_n\vartheta(t_nu_n)+f(z,t_nu_n)](t_nu_n)d+c_{25}\nonumber\\
		&&\mbox{for some}\ c_{25}>0,\ \mbox{and all}\ n\geq n_2\ (\mbox{see (\ref{eq42})})\nonumber\\
		&&\leq\int_\Omega p[\lambda_n\theta(t_nu_n)+F(z,t_nu_n)]dz+\int_\Omega e_{\lambda_n}(z,u_n)dz+c_{26}\nonumber\\
		&&\mbox{for some}\ c_{26}>0,\ \mbox{and all}\ n\geq n_2\ (\mbox{see $H(f)(iii)$}),\nonumber\\
		&\Rightarrow&p\Im_{\lambda_n}(t_nu_n)\leq c_{27}\ \mbox{for some}\ c_{27}>0,\ \mbox{and all}\ n\geq n_2\ (\mbox{see (\ref{eq49})}).
	\end{eqnarray}
	
	Comparing (\ref{eq59}) and (\ref{eq60}), we obtain a contradiction.
	
	We conclude that $\{u_n\}_{n\geq 1}\subseteq W^{1,p}_0(\Omega)$ is bounded and this proves the claim.
	
	On account of the claim, we way assume
	$$u_n\stackrel{w}{\rightarrow}u_*\ \mbox{in}\ W^{1,p}_0(\Omega)\ \mbox{and}\ u_n\rightarrow u_*\ \mbox{in}\ L^r(\Omega).$$
	
	We return to (\ref{eq46}) and use $h=u_n-u_*\in W^{1,p}_0(\Omega)$. Passing to the limit as $n\rightarrow\infty$, we obtain
	\begin{eqnarray}\label{eq61}
		&&\lim\limits_{n\rightarrow\infty}\left\langle A(u_n),u_n-u_*\right\rangle=0,\nonumber\\
		&\Rightarrow&u_n\rightarrow u_*\ \mbox{in}\ W^{1,p}_0(\Omega)\ \mbox{as}\ n\rightarrow\infty\ (\mbox{see Lemma \ref{lem4}}).
	\end{eqnarray}
	
	If in (\ref{eq46}) we pass to the limit as $n\rightarrow\infty$ and use (\ref{eq61}), then we obtain
	\begin{equation}\label{eq62}
		\left\langle A(u_*),h\right\rangle=\int_\Omega j_{\lambda^*}(z,u_*)hdz\ \mbox{for all}\ h\in W^{1,p}_0(\Omega).
	\end{equation}
	
	Recall that
	\begin{eqnarray*}
		&&\bar{u}_{\lambda_1}\leq u_n\ \mbox{for all}\ n\in\NN\ (\mbox{see Proposition \ref{prop13} and \ref{prop14}}),\\
		&\Rightarrow&\bar{u}_{\lambda_1}\leq u_*\ (\mbox{see (\ref{eq61})}),\\
		&\Rightarrow&u_*\in S_{\lambda^*}\subseteq {\rm int}\, C_+\ (\mbox{see (\ref{eq62}) and (\ref{eq42})})\ \mbox{and so}\ \lambda^*\in\mathcal{L}.
	\end{eqnarray*}
The proof is now complete.
\end{proof}

We have proved that
$$\mathcal{L}=\left(0,\lambda^*\right].$$

Next, we show that for $\lambda\in(0,\lambda^*)$, problem \eqref{eqp} has at least two positive smooth solutions.
\begin{prop}\label{prop17}
	If hypotheses $H(a),H(\vartheta),H(f)$ hold and $\lambda\in(0,\lambda^*)$, then problem \eqref{eqp} admits at least two positive solutions
	$$u_0,\hat{u}\in {\rm int}\, C_+,\ u_0\neq\hat{u}.$$
\end{prop}
\begin{proof}
	Let $0<\mu<\lambda<\eta<\lambda^*$. According to Proposition \ref{prop15}, there exist $u_0\in S_\lambda\subseteq {\rm int}\, C_+$ and $u_\mu\in S_\mu\subseteq {\rm int}\, C_+$, $u_\eta\in S_\eta\subseteq {\rm int}\, C_+$ such that
	\begin{eqnarray}\label{eq63}
		&&u_0-u_\mu\in {\rm int}\, C_+\ \mbox{and}\ u_\mu-u_0\in {\rm int}\, C_+,\nonumber\\
		&\Rightarrow&u_0\in {\rm int}_{C^1_0(\overline{\Omega})}[u_\mu,u_\eta].
	\end{eqnarray}
	
	We introduce the following truncations of the reaction in problem \eqref{eqp}:
	\begin{eqnarray}
		&&i_\lambda(z,x)=\left\{\begin{array}{ll}
			\lambda\vartheta(u_\mu(z))+f(z,u_\mu(z))&\mbox{if}\ x\leq u_\mu(z)\\
			\lambda\vartheta(x)+f(z,x)&\mbox{if}\ u_\mu(z)<x
		\end{array}\right.\label{eq64}\\
		&\mbox{and}&\hat{i}_\lambda(z,x)=\left\{\begin{array}{ll}
			i_\lambda(z,x)&\mbox{if}\ x\leq u_\eta(z)\\
			i_\lambda(z,\eta(z))&\mbox{if}\ u_\eta(z)<x.
		\end{array}\right.\label{eq65}
	\end{eqnarray}
	
	Both are Carath\'eodory functions. We set
	$$I_\lambda(z,x)=\int^x_0 i_\lambda(z,s)ds\ \mbox{and}\ \hat{I}_\lambda(z,x)=\int^x_0\hat{i}_\lambda(z,s)ds$$
	and consider the $C^1$-functionals $\sigma_\lambda,\hat{\sigma}_\lambda:W^{1,p}_0(\Omega)\rightarrow\RR$ defined by
	\begin{eqnarray*}
		&&\sigma_\lambda(u)=\int_\Omega G(Du)dz-\int_\Omega I_\lambda(z,u)dz\\
		&\mbox{and}&\hat{\sigma}_\lambda(u)=\int_\Omega G(Du)dz-\int_\Omega\hat{I}_\lambda(z,u)dz\ \mbox{for all}\ u\in W^{1,p}_0(\Omega).
	\end{eqnarray*}
	
	Using (\ref{eq64}) and (\ref{eq65}), as in the proof of Proposition \ref{prop15}, we show that
	\begin{equation}\label{eq66}
		K_{\sigma_\lambda}\subseteq\left[u_\mu\right)\cap {\rm int}\, C_+\ \mbox{and}\ K_{\hat{\sigma}_\lambda}\subseteq[u_\mu,u_\eta]\cap {\rm int}\, C_+.
	\end{equation}
	
	Evidently, we may assume that
	\begin{equation}\label{eq67}
		K_{\sigma_\lambda}\cap[u_\mu,u_\eta]=\{u_0\}.
	\end{equation}
	
	Otherwise, from (\ref{eq66}) and (\ref{eq64}) we see that we already have a second positive smooth solution and so we are done. From Corollary \ref{cor3} and (\ref{eq65}), it is clear that $\hat{\sigma}_\lambda(\cdot)$ is coercive. Also by the Sobolev embedding theorem, $\hat{\sigma}(\cdot)$ is sequentially weakly lower semicontinuous. So, we can find $\hat{u}_0\in W^{1,p}_0(\Omega)$ such that
	\begin{eqnarray*}
		&&\hat{\sigma}_\lambda(\hat{u}_0)=\inf\{\hat{\sigma}_\lambda(u):u\in W^{1,p}_0(\Omega)\},\\
		&\Rightarrow&\hat{u}_0\in K_{\hat{\sigma}_\lambda}.
	\end{eqnarray*}
	
	Note that $\sigma'_\lambda=\hat{\sigma}'_\lambda$ on $[u_\mu,u_\eta]$ (see (\ref{eq64}) and (\ref{eq65})). Therefore from (\ref{eq66}) and (\ref{eq67}) it follows that
	\begin{eqnarray}\label{eq68}
		&&\hat{u}_0=u_0,\nonumber\\
		&\Rightarrow&u_0\ \mbox{is a local}\ C^1_0(\overline{\Omega})-\mbox{minimizer of}\ \sigma_\lambda(\cdot)\nonumber\\
		&&(\mbox{see (\ref{eq63}) and note that}\ \sigma_\lambda=\hat{\sigma}_\lambda\ \mbox{on}\ [u_\mu,u_\eta]),\nonumber\\
		&\Rightarrow&u_0\ \mbox{is a local}\ W^{1,p}_0(\Omega)-\mbox{minimizer of}\ \sigma_\lambda(\cdot)\\
		&&(\mbox{see Papageorgiou \& R\u{a}dulescu \cite[Proposition 8]{25}}).\nonumber
	\end{eqnarray}
	
	From (\ref{eq64}) and (\ref{eq66}), it is clear that we may assume that
	\begin{equation}\label{eq69}
		K_{\sigma_\lambda}\ \mbox{is finite.}
	\end{equation}
	Otherwise on account of (\ref{eq66}), we see that we already have an infinity of positive smooth solutions and so we are done. Then from (\ref{eq68}) and (\ref{eq69}) it follows that we can find small enough $\rho\in(0,1)$ such that
	\begin{equation}\label{eq70}
		\sigma_\lambda(u_0)<\inf\{\sigma_\lambda(u):||u-u_0||=\rho\}=m_\lambda
	\end{equation}
	(see Aizicovici, Papageorgiou \& Staicu \cite{1}, proof of Proposition 29). Also, on account of hypotheses $H(f)(ii),(iii)$ for $u\in {\rm int}\, C_+$, we have
	\begin{equation}\label{eq71}
		\sigma_\lambda(tu)\rightarrow-\infty\ \mbox{as}\ t\rightarrow+\infty.
	\end{equation}
	
	We will show that
	\begin{equation}\label{eq72}
		\sigma_\lambda(\cdot)\ \mbox{satisfies the C-condition}.
	\end{equation}
	
	To this end, we consider a sequence $\{u_n\}_{n\geq 1}\subseteq W^{1,p}_0(\Omega)$ such that
	\begin{eqnarray}
		&&|\sigma_\lambda(u_n)|\leq c_{28}\ \mbox{for some}\ c_{28}>0,\ \mbox{all}\ n\in\NN,\label{eq73}\\
		&&(1+||u_n||)\sigma'_\lambda(u_n)\rightarrow 0\ \mbox{in}\ W^{-1,p'}(\Omega)\ \mbox{as}\ n\rightarrow\infty.\label{eq74}
	\end{eqnarray}
	
	From (\ref{eq74}) we have
	\begin{eqnarray}\label{eq75}
		&&\left|\left\langle A(u_n),h\right\rangle-\int_\Omega i_\lambda(z,u_n)hdz\right|\leq\frac{\epsilon_n||h||}{1+||u_n||}\\
		&&\mbox{for all}\ h\in W^{1,p}_0(\Omega),\ \mbox{with}\ \epsilon_n\rightarrow 0^+.\nonumber
	\end{eqnarray}
	
	In (\ref{eq75}) we choose $h=-u^-_n\in W^{1,p}_0(\Omega)$ and obtain
	\begin{eqnarray}\label{eq76}
		&&\frac{c_1}{p-1}||Du^-_n||^{p-1}_p\leq c_{29}\ \mbox{for some}\ c_{29}>0,\ \mbox{and all}\ n\in\NN\ (\mbox{see Lemma \ref{lem2} and (\ref{eq64})}),\nonumber\\
		&\Rightarrow&\{u^-_n\}_{n\geq 1}\subseteq W^{1,p}_0(\Omega)\ \mbox{is bounded}.
	\end{eqnarray}
	
	From (\ref{eq73}) and (\ref{eq76}) it follows that
	\begin{equation}\label{eq77}
		\int_\Omega pG(Du^+_n)dz-\int_\Omega pI_\lambda(z,u^+_n)dz\leq c_{30}\ \mbox{for some}\ c_{30}>0,\ \mbox{and all}\ n\in\NN.
	\end{equation}
	
	In (\ref{eq75}) we choose $h=u^+_n\in W^{1,p}_0(\Omega)$. Then
	\begin{equation}\label{eq78}
		-\int_\Omega(a(Du^+_n),Du^+_n)_{\RR^N}dz+\int_\Omega i_\lambda(z,u^+_n)u^+_ndz\leq\epsilon_n\ \mbox{for all}\ n\in\NN.
	\end{equation}
	
	We add (\ref{eq77}) and (\ref{eq78}) and we use hypothesis $H(a)(iv)$. We obtain
	\begin{eqnarray}\label{eq79}
		&&\int_\Omega[i_\lambda(z,u^+_n)u^+_n-pI_\lambda(z,u^+_n)]dz\leq c_{31}\ \mbox{for some}\ c_{31}>0,\ \mbox{and all}\ n\in\NN\nonumber\\
		&\Rightarrow&\int_\Omega e_\lambda(z,u^+_n)dz\leq c_{32}\ \mbox{for some}\ c_{32}>0,\ \mbox{and all}\ n\in\NN\ (\mbox{see (\ref{eq64})}).
	\end{eqnarray}
	
	Using (\ref{eq79}) and reasoning as in the proof of Proposition \ref{prop16} (see the claim of the proof of Proposition \ref{prop16}), we show that
	\begin{eqnarray*}
		&&\{u^+_n\}_{n\geq 1}\subseteq W^{1,p}_0(\Omega)\ \mbox{is bounded},\\
		&\Rightarrow&\{u_n\}_{n\geq 1}\subseteq W^{1,p}_0(\Omega)\ \mbox{is bounded (see (\ref{eq76}))}.
	\end{eqnarray*}
	
	From this, as in the proof of Proposition \ref{prop16}, using the $(S)_+$-property of $A(\cdot)$ (see Lemma \ref{lem4}), we conclude that (\ref{eq72}) is true.
	
	On account of (\ref{eq70}), (\ref{eq71}), (\ref{eq72}), we can use Theorem \ref{th1} (mountain pass theorem). So, we can find $\hat{u}\in W^{1,p}_0(\Omega)$ such that
	\begin{equation}\label{eq80}
		\left\{\begin{array}{l}
			\hat{u}\in K_{\lambda}\subseteq\left[u_\mu\right)\cap {\rm int}\, C_+\ (\mbox{see (\ref{eq66})}),\\
			\sigma_\lambda(u_0)<m_\lambda\leq\sigma_\lambda(\hat{u})\ (\mbox{see (\ref{eq70})}).
		\end{array}\right\}
	\end{equation}
	
	From (\ref{eq80}) and (\ref{eq64}), we conclude that
	$$\hat{u}\in S_\lambda\subseteq {\rm int}\, C_+,\ \hat{u}\neq u_0.$$
The proof is now complete.
\end{proof}

\section{The minimal positive solution}

In this section we show that for every $\lambda\in\mathcal{L}=\left(0,\lambda^*\right]$ problem \eqref{eqp} has a smallest positive solution $u^*_\lambda$ and we examine the monotonicity and continuity properties of the map $\lambda\mapsto u^*_\lambda$.

First, we show that $S_\lambda$ is downward directed. This means that if $u_1,u_2\in S_\lambda$, then we can find $u\in S_\lambda$ such that $u\leq u_1,u\leq u_2$. Our proof is inspired by the proof of Lemma 4.1 of Filippakis \& Papageorgiou \cite{7}, where a similar result is proved for the upper solutions of nonlinear Dirichlet problems without singularities, driven by the $p$-Laplacian (see also Papageorgiou, R\u{a}dulescu \& Repov\v{s} \cite{26, 26bis}, for Robin problems).
\begin{prop}\label{prop18}
	If hypotheses $H(a),H(\vartheta), H(f)$ hold and $\lambda\in\mathcal{L}=\left(0,\lambda^*\right]$, then $S_\lambda$ is downward directed.
\end{prop}
\begin{proof}
	Let $u_1,u_2\in S_\lambda$ and $\epsilon>0$. We introduce the following cut-off frunction
	$$\eta_\epsilon(x)=\left\{\begin{array}{ll}
		-1&\mbox{if}\ x<-\epsilon\\
		\frac{x}{\epsilon}&\mbox{if}\ -\epsilon\leq x\leq \epsilon\\
		1&\mbox{if}\ \epsilon<x.
	\end{array}\right.$$
	
	Evidently, $\eta_\epsilon:\RR\rightarrow\RR$ is Lipschitz continuous. So, we have
	\begin{equation}\label{eq81}
		\left\{\begin{array}{l}
			\eta_\epsilon((u_1-u_2)^-(\cdot))\in W^{1,p}_0(\Omega),\\
			\nabla\eta_\epsilon((u_1-u_2)^-)=\eta'_\epsilon((u_1-u_2)^-)D((u_1-u_2)^-)
		\end{array}\right\}
	\end{equation}
	(see, for example, Gasinski \& Papageorgiou \cite{9}, Proposition 2.4.25 and Remark 2.4.26, p.~195).

Let $y\in C^1_0(\Omega)=\{y\in C^1(\Omega):y\ \mbox{has compact support in}\ \Omega\}$, $y\geq 0$. Then
	\begin{equation}\label{eq82}
		\left\{\begin{array}{l}
			\eta_\epsilon((u_1-u_2)^-)y\in W^{1,p}_0(\Omega),\\
			D(\eta((u_1-u_2)^-)y)=yD(\eta((u_1-u_2)^-))+\eta_\epsilon((u_1-u_2)^-)Dy
		\end{array}\right\}
	\end{equation}
	(see Proposition 2.4.22 of Gasinski \& Papageorgiou \cite[p. 193]{9}).	
	Since $u_1,u_2\in S_\lambda\subseteq {\rm int}\, C_+$, we have
	\begin{eqnarray}
		&&\int_\Omega(a(Du_1),Dh)_{\RR^N}dz=\int_\Omega[\lambda\vartheta(u_1)+f(z,u_1)]hdz,\label{eq83}\\
		&&\int_\Omega(a(Du_2),Dh)_{\RR^N}dz=\int_\Omega[\lambda\vartheta(u_2)+f(z,u_2)]hdz\ \mbox{for all}\ h\in W^{1,p}_0(\Omega).\label{eq84}
	\end{eqnarray}
	
	In (\ref{eq83}) we choose $h=\eta_\epsilon((u_1-u_2)^-)y\in W^{1,p}_0(\Omega)$ and in (\ref{eq84}) we choose $h=[1-\eta_\epsilon((u_1-u_2)^-)]y\in W^{1,p}_0(\Omega)$. We add the two equalities and obtain
	\begin{eqnarray}\label{eq85}
		&&\int_\Omega(a(Du),D(\eta_\epsilon((u_1-u_2)^-)y))_{\RR^n}dz+\int_\Omega(a(Du_2),D((1-\eta_\epsilon((u_1-u_2)^-))y))_{\RR^N}dz\nonumber\\
		&=&\int_\Omega[\lambda\vartheta(u_1)+f(z,u_1)]\eta_\epsilon((u_1-u_2)^-)ydz+\nonumber\\
		&&\int_\Omega[\lambda\vartheta(u_2)+f(z,u_2)](1-\eta_\epsilon((u_1-u_2)^-))ydz.
	\end{eqnarray}
	
	We examine the two summands on the left-hand side of (\ref{eq85}). We have
	\begin{eqnarray}\label{eq86}
		&&\int_\Omega(a(Du_1),D(\eta_\epsilon((u_1-u_2)^-)y))_{\RR^N}dz=\nonumber\\
		&&\int_\Omega y(a(Du_1),D(\eta_\epsilon((u_1-u_2)^-)))_{\RR^N}dz+\int_\Omega\eta_\epsilon((u_1-u_2)^-)(a(Du_1),Dy)_{\RR^N}dz=\nonumber\\
		&&-\int_{\{-\epsilon\leq u_1-u_2\leq 0\}}y(a(Du_1),D(u_1-u_2))_{\RR^N}dz+\int_\Omega\eta_\epsilon((u_1-u_2)^-)(a(Du_1),Dy)_{\RR^N}dz.
	\end{eqnarray}
	
	Similarly we have
	\begin{eqnarray}\label{eq87}
		&&\int_\Omega(a(Du_2),D((1-\eta_\epsilon((u_1-u_2)^-))y))_{\RR^N}dz\nonumber\\
		&=&\int_\Omega y(a(Du_2),D(u_1-u_2))_{\RR^N}dz+\int_\Omega[1-\eta_\epsilon((u_1-u_2)^-)](a(Du_2),Dy)_{\RR^N}dz.
	\end{eqnarray}
	
	We use (\ref{eq86}) and (\ref{eq87}) in (\ref{eq85}). Since $y\geq 0$ and exploiting the monotonicity of $a(\cdot)$, we obtain
	\begin{eqnarray}\label{eq88}
		&&\int_\Omega\eta_\epsilon((u_1-u_2)^-)(a(Du_1),Dy)_{\RR^N}dz+\int_\Omega[1-\eta_\epsilon((u_1-u_2)^-)](a(Du_2),Dy)_{\RR^N}dz\nonumber\\
		&\geq&\int_\Omega[\lambda\vartheta(u_1)+f(z,u_1)]\eta_\epsilon((u_1-u_2)^-)dz+\nonumber\\
		&&\hspace{1cm}\int_\Omega[\lambda\vartheta(u_2)+f(z,u_2)](1-\eta_\epsilon((u_1-u_2)^-))dz.
	\end{eqnarray}
	
	Note that
	$$\eta_\epsilon((u_1-u_2)^-(z))\rightarrow\chi_{\{u_1<u_2\}}(z)\ \mbox{for almost all}\ z\in\Omega,\ \mbox{as}\ \epsilon\rightarrow 0^+.$$
	
	So, if in (\ref{eq88}) we let $\epsilon\rightarrow 0^+$, then
	\begin{eqnarray}\label{eq89}
		&&\int_{\{u_1<u_2\}}(a(Du_1),Dy)_{\RR^N}dz+\int_{\{u_2\leq u_1\}}(a(Du_2),Dy)_{\RR^N}dz\nonumber\\
		&\geq&\int_{\{u_1<u_2\}}[\lambda\vartheta(u_1)+f(z,u_1)]ydz+\int_{\{u_2\leq u_1\}}[\lambda\vartheta(u_2)+f(z,u_2)]udz.
	\end{eqnarray}
	
	If $\tilde{u}=\min\{u_1,u_2\}\in W^{1,p}_0(\Omega)$, then we know that
	\begin{equation}\label{eq90}
		D\tilde{u}=\left\{\begin{array}{l}
			Du_1\ \mbox{for almost all}\ z\in\{u_1<u_2\}\\
			Du_2\ \mbox{for almost all}\ z\in\{u_2\leq u_1\}.
		\end{array}\right.
	\end{equation}
	
	Moreover, by Stampacchia's theorem we have that
	$$D(u_2-u_1)=0\ \mbox{for almost all}\ z\in\{u_1=u_2\}$$
	(see Gasinski \& Papageorgiou \cite[pp. 195-196]{9}).

From (\ref{eq89}) and (\ref{eq90}) it follows that
	$$\int_\Omega(a(D\tilde{u}),Dy)_{\RR^N}dz\geq\int_\Omega[\lambda\vartheta(\tilde{u})+f(z,\tilde{u})]ydz\ \mbox{for all}\ y\in C^1_0(\Omega),\ y\geq 0.$$
	
	The density of such $y$'s in the positive cone of $W^{1,p}_0(\Omega)$ (that is, in the set $W_+=\{u\in W^{1,p}_0(\Omega):u(z)\geq 0\ \mbox{for almost all}\ z\in\Omega\}$, see Gasinski \& Papageorgiou \cite[Proposition 1.154]{10}), implies that
	\begin{eqnarray}\label{eq91}
		&&\int_\Omega(a(D\tilde{u}),Dy)_{\RR^N}dz\geq\int_\Omega[\lambda\vartheta(\tilde{u})+f(z,\tilde{u})]ydz\\
		&&\mbox{for all}\ y\in W^{1,p}_0(\Omega)\ \mbox{with}\ y\geq 0.\nonumber
	\end{eqnarray}
	
	We consider the following truncation of the reaction in problem \eqref{eqp}:
	\begin{equation}\label{eq92}
		\mu_\lambda(z,x)=\left\{\begin{array}{ll}
			\lambda\vartheta(\bar{u}_\lambda(z))+f(z,\bar{u}_\lambda(z))&\mbox{if}\ x<\bar{u}_\lambda(z)\\
			\lambda\vartheta(x)+f(z,x)&\mbox{if}\ \bar{u}_\lambda(z)\leq x\leq \tilde{u}(z)\\
			\lambda\vartheta(\tilde{u}(z))+f(z,\tilde{u}(z))&\mbox{if}\ \tilde{u}(z)<x.
		\end{array}\right.
	\end{equation}
	
	This is a Carath\'eodory function. We set $M_\lambda(z,x)=\int^x_0\mu_\lambda(z,s)ds$ and consider the $C^1$-functional (see Proposition \ref{prop5}) $\tau_\lambda:W^{1,p}_0(\Omega)\rightarrow\RR$ defined by
	$$\tau_\lambda(v)=\int_\Omega G(Dv)dz-\int_\Omega M_\lambda(z,v)dz\ \mbox{for all}\ v\in W^{1,p}_0(\Omega).$$
	
	As before (see the proof of Proposition \ref{prop15}), using (\ref{eq92}), we can show that
	\begin{equation}\label{eq93}
		K_{\tau_\lambda}\subseteq[\bar{u}_\lambda,\tilde{u}]\cap {\rm int}\, C_+.
	\end{equation}
	
	From (\ref{eq92}) it is clear that $\tau_\lambda(\cdot)$ is coercive. In addition, $\tau_\lambda(\cdot)$ is sequentially weakly lower semicontinuous. So, we can find $u\in W^{1,p}_0(\Omega)$ such that
	\begin{eqnarray*}
		&&\tau_\lambda(u)=\inf\{\tau_\lambda(v):v\in W^{1,p}_0(\Omega)\},\\
		&\Rightarrow&u\in K_{\tau_\lambda}\subseteq[\bar{u}_\lambda,\tilde{u}]\cap {\rm int}\, C_+,\\
		&\Rightarrow&u\in S_\lambda\subseteq {\rm int}\, C_+\ \mbox{and}\ u\leq \tilde{u}\leq u_1,u_2.
	\end{eqnarray*}
	
	This proves that $S_\lambda$ is downward directed.
\end{proof}

Using this proposition, we can show the existence of a minimal positive solution for problem \eqref{eqp}.
\begin{prop}\label{prop19}
	If hypotheses $H(a),H(\vartheta),H(f)$ hold and $\lambda\in\mathcal{L}=\left(0,\lambda^*\right]$, then problem \eqref{eqp} admits a smallest positive solution $u^*_\lambda\in S_\lambda\subseteq {\rm int}\, C_+$ (that is, $u^*_\lambda\leq u$ for all $u\in S_\lambda$).
\end{prop}
\begin{proof}
	On account of Proposition \ref{prop18} and using Lemma 3.10 of Hu \& Papageorgiou \cite[p. 178]{15}, we can find $\{u_n\}_{n\geq 1}\subseteq S_\lambda\subseteq {\rm int}\, C_+$ decreasing such that
	$$\inf S_\lambda=\inf\limits_{n\geq 1}u_n.$$
	
	We have
	\begin{eqnarray}
		&&\left\langle A(u_n),h\right\rangle=\int_\Omega[\lambda\vartheta(u_n)+f(z,u_n)]hdz\ \mbox{for all}\ h\in W^{1,p}_0(\Omega),\label{eq94}\\
		&&0\leq\bar{u}_\lambda\leq u_n\leq u_1\ \mbox{for all}\ n\in\NN\ (\mbox{see Proposition \ref{prop13}}).\label{eq95}
	\end{eqnarray}
	
	Choosing $h=u_n\in W^{1,p}_0(\Omega)$ in (\ref{eq94}) and using (\ref{eq95}) and hypotheses $H(\vartheta)(ii)$ and $H(f)(i)$, we can infer that
	$$\{u_n\}_{n\geq 1}\subseteq W^{1,p}_0(\Omega)\ \mbox{is bounded}.$$
	
	So, we may assume that
	\begin{equation}\label{eq96}
		u_n\stackrel{w}{\rightarrow}u^*_\lambda\ \mbox{in}\ W^{1,p}_0(\Omega)\ \mbox{and}\ u_n\rightarrow u^*_\lambda\ \mbox{in}\ L^p(\Omega).
	\end{equation}
	
	In (\ref{eq94}) we choose $h=u_n-u^*_\lambda\in W^{1,p}_0(\Omega)$, pass to the limit as $n\rightarrow\infty$ and use (\ref{eq95}), (\ref{eq96}). Then
	\begin{eqnarray*}
		&&\lim\limits_{n\rightarrow\infty}\left\langle A(u_n),u_n-u^*_\lambda\right\rangle=0,\\
		&\Rightarrow&u_n\rightarrow u^*_\lambda\ \mbox{in}\ W^{1,p}_0(\Omega)\ (\mbox{see Lemma \ref{lem4}}).
	\end{eqnarray*}
	
	Taking the limit as $n\rightarrow\infty$ in (\ref{eq94}), we obtain
	\begin{eqnarray*}
		&&\left\langle A(u^*_\lambda),h\right\rangle=\int_\Omega[\lambda\vartheta(u^*_\lambda)+f(z,u^*_\lambda)]hdz\ \mbox{for all}\ h\in W^{1,p}_0(\Omega),\ \bar{u}_\lambda\leq u^*_\lambda,\\
		&\Rightarrow&u^*_\lambda\in S_\lambda\subseteq {\rm int}\, C_+\ \mbox{and}\ u^*_\lambda=\inf S_\lambda.
	\end{eqnarray*}
The proof is now complete.
\end{proof}

Consider the minimal solution map $\chi:\mathcal{L}=\left(0,\lambda^*\right]\rightarrow C^1_0(\overline{\Omega})$ defined by
$$\chi(\lambda)=u^*_\lambda\ \mbox{for all}\ \lambda\in\mathcal{L}.$$

The next proposition establishes the monotonicity and continuity properties of this map.
\begin{prop}\label{prop20}
	If hypotheses $H(a),H(\vartheta),H(f)$ hold, then
	\begin{itemize}
		\item[(a)] $\chi(\cdot)$ is strictly increasing (that is, $0<\lambda<\eta\leq\lambda^*\Rightarrow\chi(\eta)-\chi(\lambda)\in {\rm int}\, C_+$);
		\item[(b)] $\chi(\cdot)$ is left continuous.
	\end{itemize}
\end{prop}
\begin{proof}
	$(a)$ Let $0<\mu<\lambda\leq \lambda^*$. According to Proposition \ref{prop15}, we can find $u_\mu\in S_\mu\subseteq {\rm int}\, C_+$ such that
	\begin{eqnarray*}
		&&u^*_\lambda-u_\mu\in {\rm int}\, C_+,\\
		&\Rightarrow&u^*_\lambda-u^*_\mu\in {\rm int}\, C_+\ (\mbox{since}\ u^*_\mu\leq u_\mu),\\
		&\Rightarrow&\chi(\cdot)\ \mbox{is strictly increasing}.
	\end{eqnarray*}
	
	$(b)$ Let $\{\lambda_n\}_{n\geq 1}\subseteq\mathcal{L}$ and assume that $\lambda_n\rightarrow\lambda^-\ (\lambda\in\mathcal{L})$. We set $u^*_n=u^*_{\lambda_n}$ for all $n\in\NN$. The sequence $\{u^*_n\}_{n\geq 1} \subseteq {\rm int}\, C_+$ is increasing (see $(a)$) and $u^*_n\leq u^*_{\lambda^*}\in S_{\lambda^*}\subseteq {\rm int}\, C_+$. We have
	$$\left\langle A(u_n),h\right\rangle=\int_\Omega[\lambda_n\vartheta(u^*_n)+f(z,u^*_n)]hdz\ \mbox{for all}\ h\in W^{1,p}_0(\Omega).$$
	
	The nonlinear regularity theory of Lieberman \cite{18} implies that we can find $\alpha\in(0,1)$ and $c_{33}>0$ such that
	$$u^*_n\in C^{1,\alpha}_0(\overline{\Omega}),\ ||u^*_n||_{C^{1,\alpha}_0(\overline{\Omega})}\leq c_{33}\ \mbox{for all}\ n\in\NN.$$
	
	The compact embedding of $C^{1,\alpha}_0(\overline{\Omega})$ into $C^1_0(\overline{\Omega})$ and the monotonicity of $\{u^*_n\}_{n\geq 1}$ imply that
	\begin{equation}\label{eq97}
		u^*_n\rightarrow\tilde{u}^*\ \mbox{in}\ C^1_0(\overline{\Omega}).
	\end{equation}
	
	We claim that $\tilde{u}^*=u^*_\lambda$. Arguing by contradiction, suppose that $\tilde{u}^*\neq u^*_\lambda$. Then we can find $z_0\in\overline{\Omega}$ such that
	\begin{eqnarray*}
		&&u^*_\lambda(z_0)<\tilde{u}^*(z_0),\\
		&\Rightarrow&u^*_\lambda(z_0)<u^*_n(z_0)\ \mbox{for all}\ n\geq n_0\ (\mbox{see (\ref{eq97})}),
	\end{eqnarray*}
	which contradicts $(a)$. Therefore $\tilde{u}^*=u^*_\lambda$ and so we have
	\begin{eqnarray*}
		&&u^*_n\rightarrow u^*_\lambda\ \mbox{in}\ C^1_0(\overline{\Omega}),\\
		&\Rightarrow&\chi(\cdot)\ \mbox{is left continuous},
	\end{eqnarray*}
which concludes the proof of
Proposition \ref{prop20}.
\end{proof}

The proof of Theorem \ref{th21} is now complete.

\medskip
{\bf Acknowledgements.} This research was supported by the Slovenian Research Agency grants
P1-0292, J1-8131, J1-7025, N1-0064, and N1-0083.

\end{document}